\documentclass[12pt]{amsart}
\usepackage{a4wide,enumerate,xcolor}
\usepackage{amsmath,graphicx}
\allowdisplaybreaks

\let\pa\partial
\let\na\nabla
\let\eps\varepsilon
\newcommand{\N}{{\mathbb N}}
\newcommand{\R}{{\mathbb R}}
\newcommand{\diver}{\operatorname{div}}
\newcommand{\dv}{\mathrm{d}v}
\newcommand{\ds}{\mathrm{d}s}
\newcommand{\vv}{\langle v\rangle}
\newcommand{\red}{\textcolor{black}}
\newcommand{\blue}{\textcolor{black}}

\newtheorem{theorem}{Theorem}
\newtheorem{lemma}[theorem]{Lemma}

\newtheorem{remark}[theorem]{Remark}

\begin{document}

\title[Multispecies Fokker--Planck--Landau system]{Global weak solutions for 
a nonlocal \\ multispecies Fokker--Planck--Landau system}

\author[J. Hu]{Jingwei Hu}
\address{Department of Applied Mathematics, University of Washington, Seattle, WA 98195, USA}
\email{hujw@uw.edu}

\author[A. J\"ungel]{Ansgar J\"ungel}
\address{Institute of Analysis and Scientific Computing, Technische Universit\"at Wien,
Wiedner Hauptstra\ss e 8--10, 1040 Wien, Austria}
\email{juengel@tuwien.ac.at}

\author[N. Zamponi]{Nicola Zamponi}
\address{Institut f\"ur Angewandte Analysis, Universit\"at Ulm, Helmholtzstra{\ss}e 18,
89081 Ulm, Germany}
\email{nicola.zamponi@uni-ulm.de}

\date{\today}

\thanks{The first and second authors thank the Isaac Newton Institute for Mathematical
Sciences for support and hospitality during the programme ``Frontiers in Kinetic Theory''
when work on this paper was undertaken. 
The first author is partially supported under the NSF grant DMS-2153208, 
AFOSR grant FA9550-21-1-0358, and DOE grant DE-SC0023164.
The second and third authors acknowledge partial support from
the Austrian Science Fund (FWF), grants P33010 and F65.
This work has received funding from the European
Research Council (ERC) under the European Union's Horizon 2020 research and
innovation programme, ERC Advanced Grant no.~101018153.}

\begin{abstract}
The global-in-time existence of weak solutions to a spatially homogeneous multispecies
Fokker--Planck--Landau system for plasmas in the three-dimensional whole space is shown. 
The Fokker--Planck--Landau system is a simplification of the Landau equations
assuming a linearized, velocity-independent, and isotropic kernel. The resulting equations
depend nonlocally and nonlinearly on the moments of the distribution functions via the
multispecies local Maxwellians. The existence proof is
based on a three-level approximation scheme, energy and entropy estimates, as well as
compactness results, and it holds for both soft and hard potentials.
\end{abstract}



\maketitle


\section{Introduction}

The Fokker--Planck--Landau equations describe the local collisional relaxation process
of the particle distribution functions in plasmas under binary collisions \cite{ArBu91}.
In this paper, we investigate a multispecies, linearized, 
spatially homogeneous version of these equations. 
More precisely, the distribution functions $f_i(v,t)$ of the $i$th species
of the multicomponent plasma, depending on the velocity $v\in\R^3$ and time $t\ge 0$,
are assumed to satisfy the initial-value problem
\begin{align}\label{1.eq}
  & \pa_t f_i = \sum_{j=1}^s c_{ji}\diver\bigg(\na f_i + m_i\frac{v-u_{ji}}{T_{ji}}
	f_i\bigg)\quad\mbox{in }\R^3,\ t>0,\\
  & f_i(\cdot,0)=f_i^0\quad\mbox{in }\R^3, \ i=1,\ldots,s, \label{1.ic}
\end{align}
where $s\in\N$ is the number of species and $m_i>0$ the molar mass of the $i$th species. 
Before defining the quantities $c_{ji}$, $u_{ji}$, and $T_{ji}$, we introduce
the moments of $f_i$, namely
the number density $n_i$, partial velocity $u_i$, and partial temperature $T_i$ by
\begin{equation}\label{1.mom}
  n_i = \int_{\R^3}f_i\dv, \quad u_i = \frac{1}{n_i}\int_{\R^3}f_i v\dv, \quad
	T_i = \frac{m_i}{3n_i}\int_{\R^3}f_i|v-u_i|^2\dv,
\end{equation}
as well as the partial mass density $\rho_i=m_i n_i$. 
Then the diffusion coefficients $c_{ji}$ and ``multispecies'' velocities $u_{ji}$ and
temperatures $T_{ji}$ are given by
\begin{align}
  c_{ji} &= \frac{|\log\Lambda|q_i^2q_j^2}{8\pi\eps_0^2m_i^2} 
	n_j\bigg(\frac{T_j}{m_j}\bigg)^{\gamma/2}, 
	\label{1.cji} \\
	u_{ji} &= \frac{c_{ji}m_i\rho_i u_i + c_{ij}m_j\rho_j u_j}{c_{ji}m_i\rho_i+c_{ij}m_j\rho_j}, 
	\label{1.uji} \\
	T_{ji} &= \frac{c_{ji}\rho_i T_i + c_{ij}\rho_j T_j}{c_{ji}\rho_i + c_{ij}\rho_j} +
  \frac{c_{ji}m_i\rho_i c_{ij}m_j\rho_j |u_i - u_j|^2}{3(c_{ji}\rho_i + c_{ij}\rho_j)
  (c_{ji}m_i\rho_i + c_{ij}m_j\rho_j)}, \label{1.Tji}
\end{align}
where $\log\Lambda>0$ is the Coulomb logarithm, $\Lambda>0$ being related 
to the Debye length, $\eps_0$ is the vacuum permittivity,
$q_i$ is the charge of the $i$th species,
and $\gamma\in\R$ models the interaction strength between particles: \blue{in particular, $\gamma>2$ corresponds to hard potentials, $\gamma<2$ corresponds to soft potentials, and $\gamma=-1$ is the Coulomb interaction} (see Section \ref{sec.model} for details). 

Note that $c_{ji}$, $u_{ji}$, and $T_{ji}$ are functions of
time only, and they depend in a nonlocal and nonlinear way 
on the distribution functions. We write
$c_{ji}[f]=c_{ji}$, $u_{ji}[f]=u_{ji}$, and $T_{ji}[f]=T_{ji}$ with
$f=(f_1,\ldots,f_s)$ to make this dependence clear.
Observe that the symmetries $T_{ij}=T_{ji}$ and $u_{ji}=u_{ij}$ for $j\neq i$ 
hold as well as $T_{ii}=T_i$ and $u_{ii}=u_i$.

Single-species kinetic Fokker--Planck equations, \blue{often coupled with Vlasov equation with spatial dependence},
have been mathematically studied in the literature since 
the 1980s; see, e.g., \cite{Deg86}. One main interest was the proof of hypocoercivity
\cite{DMS15,Vil09}. There are only a few works concerned
with multispecies models. The diffusion limit of a kinetic Fokker--Planck system
for charged particles towards the Nernst--Planck equations was proved in \cite{WLL15}.
Furthermore, in \cite{FiNe22,Her16}, the limit of vanishing electron--ion mass ratios 
for nonhomogeneous kinetic Fokker--Planck systems was investigated. The
multispecies modeling in \cite{FiNe22} is very close to ours, but the model of \cite{FiNe22}
also includes spatial and electric effects. However, an existence analysis of multispecies
Fokker--Planck systems, \blue{even in the spatially homogeneous case}, is missing in the literature. In this paper, up to our knowledge,
we provide such an analysis for the first time.

Equations \eqref{1.eq}--\eqref{1.Tji} are a simplification of the Fokker--Planck--Landau
system (see Section \ref{sec.model}). In this context, the right-hand side
of \eqref{1.eq} can be interpreted as the collision operator
$$
  Q_{ji}(f_i) = \sum_{j=1}^s c_{ji}\diver\bigg(\na f_i + m_i\frac{v-u_{ji}}{T_{ji}}
	f_i\bigg).
$$
Our model satisfies some physical properties, like mass, momentum, and
energy conservation (see Lemma \ref{lem.cons} in Section \ref{sec.model}),
$$
  \frac{\mathrm{d}}{\mathrm{d}t}\int_{\R^3}(m_if_i + m_jf_j)\mu(v)\dv = 0 \quad\mbox{for }
	\mu(v)=1,v,|v|^2,
$$
and it fulfills an H-theorem or the entropy decay 
(see Lemma \ref{lem.ent} in Section \ref{sec.model}),
$$
  \frac{\mathrm{d}}{\mathrm{d}t}\sum_{i=1}^s\int_{\R^3}f_i\log f_i\dv
	= -\sum_{i,j=1}^s\int_{\R^3}c_{ji}f_i\bigg|\na\log\frac{f_i}{M_{ij}}\bigg|^2\dv \le 0,
$$
which follows from the gradient-flow-type formulation of \eqref{1.eq},
\begin{equation}\label{1.gradflow}
  \pa_t f_i = \sum_{j=1}^s c_{ji}\diver\bigg(f_i\na\log\frac{f_i}{M_{ij}}\bigg)
	\quad\mbox{in }\R^3,\ t>0,\ i=1,\ldots,s,
\end{equation}
where $M_{ij}$ are the ``multispecies'' Maxwellians
\begin{equation}\label{1.Mji}
  M_{ij}(v) = n_i\bigg(\frac{m_i}{2\pi T_{ij}}\bigg)^{3/2}
	\exp\bigg(-\frac{m_i|v-u_{ij}|^2}{2T_{ij}}\bigg).
\end{equation}
Based on these properties, we are able to prove the global existence of weak solutions 
to \eqref{1.eq}--\eqref{1.Tji}. To simplify the notation, we set $\vv:=(1+|v|^2)^{1/2}$.

\begin{theorem}\label{thm.ex}
Let $f_i^0\in L^1(\R^3;\vv^2\dv)$ be nonnegative with 
$\int_{\R^3}f_i^0\log f_i^0\dv<\infty$, let $\gamma\in\R$,
and let the constants $m_i,q_i,\Lambda,\eps_0>0$ for $i=1,\ldots,s$. Then, for any $T>0$,
there exists a nonnegative weak solution $f_i$ to \eqref{1.eq}--\eqref{1.Tji} satisfying
for all $i=1,\ldots,s$, 
\begin{align*}
  & f_i\in L^\infty(0,T;L^1(\R^3;\vv^2\dv))\cap L^2(0,T;H^1(\R^3)), \\
	& f_i\log f_i\in L^\infty(0,T;L^1(\R^3)), \quad \pa_tf_i\in L^1(0,T;W^{-1,1}(\R^3)).
\end{align*}
Moreover, there exists a constant $c>0$ such that
$T_{ji}(t)\ge c>0$, \textcolor{black}{$c_{ji}(t)\ge c>0$} for $t\in(0,T)$ and 
$c_{ji}\in L^\infty(0,T)$, $u_{ji}\in L^q(0,T)$ for any $q<\infty$. 
\end{theorem}

For the proof, we show first the existence of solutions to an
approximate problem, derive estimates uniform in the approximation parameters, and
then pass to the limit of vanishing parameters using compactness arguments. The construction
of the approximate scheme is surprisingly delicate, and we need three approximation levels.
First, we solve a regularized version of \eqref{1.eq} in the ball $B_M$ around the origin
with radius $M>0$ to avoid compactness issues due to the whole space $\R^3$. Second,
we truncate the nonlocal terms with the parameter $\eps>0$
in such a way that $c_{ji}[f]$ and $T_{ji}[f]$ are 
positive and bounded from \textcolor{black}{below} and $|u_{ji}[f]|$ is bounded from above.
Third, we need an elliptic regularization yielding $W^{1,p}(\R^3)$ solutions
with $p>3$ and a moment regularization yielding estimates for higher-order moments,
both with the same parameter $\delta>0$. 
More precisely, we add to the right-hand side of the truncated system
the expressions
$$
  E_1 = \delta\diver(|\na f_i|^{p-2}\na f_i), \quad
	E_2 = -\delta\vv^{K}f_i + \delta g(v)\int_{B_M}\vv^{K}f_i\dv,
$$
where $g(v)=\pi^{-3/2}e^{-|v|^2}$ satisfies $\int_{\R^3}g(v)\dv=1$, and $p>3$ and $K>2$
are sufficiently large. Expression $E_1$
yields an estimate for $\na f_i$ in $L^p(\R^d)$, while expression $E_2$ provides
an estimate for $f_i$ in $L^1(\R^3;\vv^{K}\dv)$. The latter term is
constructed in such a way that the mass is controlled (and conserved when $B_M$ is replaced
by $\R^3$ in the limit $M\to\infty$), since
$$
  \int_{B_M}E_2\dv = -\delta\bigg(1 - \int_{B_M}g(v)\dv\bigg)\int_{B_M}\vv^{K}f_i\dv \le 0.
$$
However, this regularization provides additional terms when using 
the test functions $f_i$, $\log f_i$, and $|v|^2$ to derive
bounds for the $L^2(\R^3)$ norm, the entropy, and the energy. 
For instance, using the test function $f_i$ in the approximated
system (see \eqref{2.approx} below), we infer from $c_{ji}[f]\ge\eps$ after some
computations, detailed in Section \ref{sec.ex}, that
\begin{align*}
  \frac12\frac{\mathrm{d}}{\mathrm{d}t}\int_{\R^3}f_i^2\dv
	&+ \delta\int_{\R^3}\vv^K f_i^2\dv + \delta\int_{\R^3}|\na f_i|^p\dv
	+ \eps\int_{\R^3}|\na f_i|^2\dv \\
	&\le C(\eps)\int_{\R^3}f_i^2\dv	+ \delta\int_{\R^3}\vv^K f_i\dv.
\end{align*}
In order to bound the last term on the right-hand side, we use (a cutoff version of)
the test function $\vv^\theta$ for $0<\theta<1-3/p$, 
which gives bounds for higher-order moments depending
on $\delta$. This is sufficient to pass to the limit $M\to\infty$ and then $\eps\to 0$.
\textcolor{black}{We point out that for the limit $M\to\infty$ we first need to show that the solution $f_i$ is nonnegative, then that the mass is nonincreasing in time, and finally we derive uniform bounds for $f_i$ in weighted Lebesgue spaces and for $\nabla f_i$ in $L^2\cap L^p$. Aubin-Lions Lemma coupled with a Cantor diagonal argument yield strong convergence of a subsequence in $L^2(B\times (0,T))$ for every $B\subset\R^3$ bounded; the uniform bound on a moment of $f_i$ allow us to infer the strong convergence of the subsequence in $L^2(\R^3\times (0,T))$.}
For the limit $\delta\to 0$, we derive uniform estimates for the entropy and energy as well as
the higher-order moment bound $\delta\int_{\R^3}\vv^{K+2}f_i\dv \le C$,
where the constant $C>0$ only depends on the initial entropy and energy. This is sufficient
to show that $E_2\to 0$ as $\delta\to 0$. 

Another issue is the limit $\delta\to 0$ in the collision operator since it
requires uniform bounds for the nonlocal terms $c_{ji}[f]$, $T_{ji}[f]$, and $u_{ji}[f]$.
The most delicate point is the proof of a uniform positive lower bound for the
temperature $T_{ji}[f]$. The idea is to estimate $T_{ji}[f]\ge\min\{T_i,T_j\}$ and
$$
  T_{i} \ge C\int_{\{|v-u_i|>\lambda\}}f_i|v-u_i|\dv
	\ge C\lambda^2\int_{\{|v-u_i|>\lambda\}}f_i\dv
	\ge C\lambda^2\bigg(n_i - \int_{\{|v-u_i|<\lambda\}}f_i\dv\bigg),
$$
where $\lambda>0$ is arbitrary.
By the Fenchel--Young inequality, we can estimate the integral on the right-hand side
in terms of the initial entropy plus a number, and a suitable choice of the parameters
allows us to conclude a lower bound only depending on the initial entropy; see Lemma 
\ref{lem.lowerT}.

Because of the truncations, we need to perform the limits $M\to\infty$, $\eps\to 0$,
and $\delta\to 0$ separately. Indeed, the energy conservation property of the 
collision operator holds only on the level of the nontruncated quantities $c_{ji}$,
$T_{ji}$, and $u_{ji}$. Therefore, we pass to the limit $\eps\to 0$ before deriving
the energy and entropy bounds that eventually allow us to perform the limit $\delta\to 0$.

Let us discuss some extensions of Theorem \ref{thm.ex}.
Our existence result also holds in the $d$-dimensional space. In this case, we choose $p>d$
and adjust the parameters $\theta>0$ and $K>2$ in a suitable way. 
We may also assume more general
functions $c_{ji}[f]$, $u_{ji}[f]$, and $T_{ji}[f]$. 
It is possible to generalize the dependency of $c_{ji}[f]$ on $T_j$, but a suitable
growth condition is needed. 
The choice of $u_{ji}[f]$ and $T_{ji}[f]$ guarantees momentum
and energy conservation (see Section \ref{sec.FPL}), and their definitions need to be
compatible with these conservation properties.


\textcolor{black}{There is no significant technical difference between the cases $\gamma\geq 2$ (hard potentials) and $\gamma < 2$ (soft potentials), nor between the cases $\gamma>0$ and $\gamma<0$. As a matter of fact, the derivation of \eqref{1.eq} destroys the singularity of the kernel in the Landau equation, meaning that choosing any value of $\gamma\in\R$ would not bring any additional technical difficulty besides the need to change the truncation in the temperature in the approximated system. }


\textcolor{black}{
It is not difficult to see that the solution $f_i(t)$ converges to a Maxwellian distribution as $t\to\infty$:
\begin{align}\label{Mstar}
f_i(t)\to M_{i}^*(v) = n_i\bigg(\frac{m_i}{2\pi T^*}\bigg)^{3/2}
\exp\bigg(-\frac{m_i|v-u^*|^2}{2T^*}\bigg)\qquad\mbox{strongly in }L^1(\R^3,(1+|v|^2)\dv),
\end{align}
where $u^* = \lim_{t\to\infty}u_{i}(t) = \lim_{t\to\infty}u_{ij}(t)$,
$T^* = \lim_{t\to\infty}T_{i}(t) = \lim_{t\to\infty}T_{ij}(t)$.
In particular the multispecies momentum and temperature become independent of the species, as $t\to\infty$. This is seen be considering the H-theorem and employing the fact that $c_{ji}(t)\geq \alpha>0$ for $t>0$ (we point out that the lower bound for the temperatures $T_i(t)$ is independent of the final time; check the proof of Lemma 10):
\begin{align*}
\frac{\mathrm{d}}{\mathrm{d}t}\sum_{i=1}^s\int_{\R^3}f_i\log f_i\dv
&= -\sum_{i,j=1}^s\int_{\R^3}c_{ji}f_i\bigg|\na\log\frac{f_i}{M_{ij}}\bigg|^2\dv
= -4\sum_{i,j=1}^s c_{ji}\int_{\R^3}\bigg|\na\sqrt{\frac{f_i}{M_{ij}}}\bigg|^2
M_{ij}\dv .
\end{align*}
since $M_{ij}$ can be transformed via a scaling into a normalized Gaussian distribution, one can apply a log-Sobolev inequality (with Gaussian weight) and infer
\begin{align*}
\sum_{i,j=1}^s c_{ji}\int_{\R^3}\bigg|\na\sqrt{\frac{f_i}{M_{ij}}}\bigg|^2
M_{ij}\dv \geq \alpha_1 C_{LS}\sum_{i,j=1}^s \int_{\R^3}f_i\log\frac{f_i}{M_{ij}} \dv ,\qquad\alpha_1>0.
\end{align*}
Integrating the entropy balance in time yields therefore
\begin{align*}
\int_0^\infty\int_{\R^3}\sum_{i,j=1}^s \left(
\frac{f_i}{M_{ij}}\log\frac{f_i}{M_{ij}} - \frac{f_i}{M_{ij}} + 1\right) M_{ij}\dv dt =
\int_0^\infty\int_{\R^3}\sum_{i,j=1}^s f_i\log\frac{f_i}{M_{ij}} \dv dt < \infty,
\end{align*}
implying (as the integrand is nonnegative) that a sequence $t_n\nearrow\infty$ of time instants exists such that 
\begin{align*}
\int_{\R^3}\sum_{i,j=1}^s \left(
\frac{f_i}{M_{ij}}\log\frac{f_i}{M_{ij}} - \frac{f_i}{M_{ij}} + 1\right) M_{ij}\dv\Big\vert_{t=t_n}\to 0.
\end{align*}
Let $f^{(n)} = f(\cdot,t_n)$, $M_{ij}^{(n)} = M_{ij}(\cdot,t_n)$.
Since $M_{ij}$ can be bounded from below via a normalized Gaussian distribution, we deduce that $f^{(n)}_i - M^{(n)}_{ij}\to 0$ a.e.~in $\R^3$. Given the bounds for $f^{(n)}_i\log f_i^{(n)}$, $f^{(n)}_i |v|^m$ in $L^1(\R^3)$ (for some $m>2$), we infer by dominated convergence that  $f^{(n)}_i - M^{(n)}_{ij}\to 0$ strongly in $L^1(\R^3,(1+|v|^2)\dv)$.
In particular the quantities $u^{(n)}_{ij}$, $T^{(n)}_{ij}$ converge to limits that are independent of $j$. Given the definition of $u^{(n)}_{ij}$, $T^{(n)}_{ij}$, it follows that the limits of the moments $u^{(n)}_i$, $T^{(n)}_i$ must be independent of $i$, yielding $f^{(n)}_i \to M^*_i$ strongly in $L^1(\R^3,(1+|v|^2)\dv)$. However, the entropy 
\begin{align*}
\sum_{i=1}^s\int_{\R^3}f_i\log f_i\dv = 
\sum_{i=1}^s\int_{\R^3}\left(f_i\log \frac{f_i}{M_i^*} - f_i + M_i^*\right)\dv + C
\end{align*}
is nonincreasing in time, hence \eqref{Mstar} holds.
Thus for long times the dynamics of the species become decoupled.
}



\textcolor{black}{
Concerning the regularity of the solutions, given the structure of the equation (in particular, the fact that the coefficients $c_{ji}$, $u_{ji}$, $T_{ji}$ are independent of $v$) we expect the solutions to be $C^\infty$ in the variable $v$ for positive times (and regular for nonnegative times as long as the initial datum is smooth enough). The time regularity issue is less trivial. However, the strictly positive lower bound for the temperatures and the coefficients $c_{ji}$, as well as the boundedness of all the coefficients, lead us to thinking that it would be possible to prove higher time regularity results by iteratively differentiating the equation in time and proving bounds for the coefficients by exploiting the regularity properties already shown in the previous step of the iteration procedure. Such bounds for the coefficients would be then employed to prove estimates for the solution in $H^k(0,T; H^m(\R^3))$ spaces. We expect this bootstrap argument to yield $C^\infty$ space-time regularity for the solution as long as the initial datum is $H^\infty$ and quickly decaying at infinity. \newline
While the uniqueness of smooth solutions originating from regular initial data would be straightforward, the uniqueness of weak solutions in the case of nonsmooth initial data is unclear. We speculate that applying the ``relative entropy method'' might yield uniqueness for general weak solutions, i.e.~considering the quantity (relative entropy)
\begin{align*}
H(f\vert g) = \int_{\R^3}f\log(f/g) \dv,
\end{align*}
where $f$, $g$ are two solutions with the same initial datum, differentiating it in time and trying to show via a Gronwall-like argument that $H(f(t)|g(t))=0$ for every $t>0$, implying $f(t)=g(t)$ for $t>0$. The argument is by no means straightforward, as we are dealing with coefficients depending on the solution itself, but nonetheless the method might work since it mirrors the entropy structure of the system.
}



\textcolor{black}{
Finally, we believe that performing a similar existence analysis in the spatially inhomogeneous case should be doable without excessive difficulties, as the techniques employed in the proof extend also to the aforementioned case. Also, we believe that it should be possible to prove existence of H solutions to the multi-species full Landau system \eqref{1.eq2}, \eqref{1.landau} by adapting the approach of Villani \cite{Vil98} to the multi-species case. Both issues, though interesting, are beyond the scope of the present paper and possibly the subject of future investigation.
}

The paper is organized as follows. Some details on the physical assumptions leading
to model \eqref{1.eq}--\eqref{1.Tji} are given in Section \ref{sec.model}. Section
\ref{sec.ex} is devoted to the proof of Theorem \ref{thm.ex}. A compactness result 
in $\R^3$ is shown in Appendix \ref{sec.comp}, 
and the rigorous treatment of nonintegrable test functions
is sketched in Appendix \ref{sec.app}.


\section{Motivation of the model and some properties}\label{sec.model}

In this section, we motivate the Fokker--Planck--Landau system \eqref{1.eq} and
detail the underlying physical assumptions leading to this model. Moreover, we
discuss its conservation properties and the H-theorem (entropy decay).

\subsection{The homogeneous Fokker--Planck--Landau system}

Model \eqref{1.eq}--\eqref{1.Tji} is a simplification of the spatially homogeneous 
multispecies Landau system by linearizing the Landau collision operator and
assuming that the operator kernel is independent of the velocity. More precisely, let 
\begin{equation}\label{1.eq2}
  \pa_t f_i = \sum_{j=1}^s \widehat{Q}_{ji}(f_j,f_i)\quad\mbox{in }\R^3,\ t>0,\ i=1\ldots,s,
\end{equation}
be the spatially homogeneous Landau equation \cite{ChCo52} for a plasma consisting of
$s$ species. The Landau collision operator
$\widehat{Q}_{ji}(f_j,f_i)$ models binary collisions between species $j$ and $i$:
\begin{equation}\label{1.landau}
  \widehat{Q}_{ji}(f_j,f_i) = \widehat{c}_{ji}\operatorname{div}_v\bigg\{\int_{\R^3}
	A(v-v_*)\bigg(f_j(v_*)\na_v f_i(v)
	- \frac{m_i}{m_j}f_i(v)\na_{v_*}f_j(v_*)\bigg)\mathrm{d}v_*\bigg\},
\end{equation}
where $\widehat{c}_{ji}=|\log\Lambda| q_i^2q_j^2/(8\pi\eps_0^2m_i^2)$ is a constant 
and $A(z)=|z|^{\beta+2}(\mathbb{I}-z\otimes z/|z|^2)$ is the (positive semidefinite)
kernel matrix with $\mathbb{I}$ being the $3\times 3$ identity matrix. 
The parameter $\beta$ refers to the case of hard potentials if $\beta>0$, Maxwellian
molecules if $\beta=0$, and soft potentials if $\beta<0$. The latter case includes
Coulomb interactions with $\beta=-3$. 
The Landau equation is obtained as the grazing collisions limit of the Boltzmann
equation \cite{ArBu91,Des92,Vil98}. 
A spectral-gap analysis for the multispecies Landau system was performed in
\cite{GuZa17}. We also refer to this reference for results on the well-posedness of
the single-species equation.

The collision operator $\widehat{Q}_{ji}$ conserves mass, momentum, and energy. 
Indeed, it can be written in the weak form
\begin{align}\label{1.weak1}
  \int_{\R^3}\widehat{Q}_{ji}(f_j,f_i)\phi\dv
	&= -\widehat{c}_{ji}\int_{\R^3}\int_{\R^3}\na_v\phi(v)^T	A(v-v_*) \\
	&\phantom{xx}{}\times\bigg(\na_v\log f_i(v)-\frac{m_i}{m_j}\na_{v_*}\log f_j(v_*)\bigg)
	f_i(v)f_j(v_*)\dv\dv_* \nonumber
\end{align}
for suitable test functions $\phi$. We obtain mass conservation by choosing $\phi=1$:
$$
  \int_{\R^3}\widehat{Q}_{ji}(f_j,f_i)\dv = 0, \quad i,j=1,\ldots,s.
$$
Using $\widehat{c}_{ji}m_i/m_j=\widehat{c}_{ij}m_j/m_i$
and exchanging $v$ and $v_*$, a computation shows that
\begin{align}\label{1.weak2}
  \int_{\R^3}\widehat{Q}_{ij}(f_i,f_j)\psi\dv
	&= \widehat{c}_{ji}\frac{m_i}{m_j}\int_{\R^3}\int_{\R^3}\na_{v_*}\psi(v_*)^T A(v-v_*) \\
	&\phantom{xx}{}\times \bigg(\na_v\log f_i(v)-\frac{m_i}{m_j}\na_{v_*}\log f_j(v_*)\bigg)
	f_i(v)f_j(v_*)\dv\dv_* \nonumber 
\end{align}
for another test function $\psi$, and an addition of \eqref{1.weak1} and \eqref{1.weak2} gives
\begin{align*}
  \int_{\R^3}&\big(\widehat{Q}_{ji}(f_j,f_i)\phi + \widehat{Q}_{ij}(f_i,f_j)\psi\big)\dv
	= -\widehat{c}_{ji}\int_{\R^3}\int_{\R^3}
	\bigg(\na_v\phi(v)-\frac{m_i}{m_j}\na_{v_*}\psi(v_*)\bigg)^T \\
	&{}\times A(v-v_*)\bigg(\na_v\log f_i(v)-\frac{m_i}{m_j}\na_{v_*}\log f_j(v_*)\bigg)
	f_i(v)f_j(v_*)\dv\dv_*.
\end{align*}
Then conservation of momentum follows by choosing $\phi(v)=m_iv$ and $\psi(v)=m_jv$,
$$
  \int_{\R^3}\widehat{Q}_{ji}(f_j,f_i)m_iv\dv 
	+ \int_{\R^3}\widehat{Q}_{ij}(f_i,f_j)m_jv\dv = 0;
$$
conservation of energy follows after the choice $\phi(v)=m_i|v|^2$ and $\psi(v)=m_j|v|^2$,
$$
  \int_{\R^3}\widehat{Q}_{ji}(f_j,f_i)m_i|v|^2\dv 
	+ \int_{\R^3}\widehat{Q}_{ij}(f_i,f_j)m_j|v|^2\dv = 0;
$$
and we obtain entropy decay after choosing $\phi(v)=\log f_i(v)$ and $\psi(v)=\log f_j(v)$:
$$
  \int_{\R^3}\widehat{Q}_{ji}(f_j,f_i)\log f_i\dv + \int_{\R^3}\widehat{Q}_{ij}(f_i,f_j)\log f_j\dv \le 0,
	\quad i,j=1,\ldots,s.
$$ 


\subsection{The homogeneous linearized Fokker--Planck--Landau system}\label{sec.FPL}

In this section, we derive model \eqref{1.eq}--\eqref{1.Tji} from the full multi-species Landau system presented in the previous section. 
Our derivation is motivated by \cite{HHM17}, where a multi-species BGK model is obtained from the multi-species Boltzmann equation.
We make two simplifications in model \eqref{1.eq2}--\eqref{1.landau}. 
First, we replace $f_j$ in $\widehat{Q}_{ji}(f_j,f_i)$ by the Maxwellian
$$
  M_{ji} = n_j\bigg(\frac{m_j}{2\pi T_{ji}}\bigg)^{3/2}\exp\bigg(
	-\frac{m_j|v-u_{ji}|^2}{2T_{ji}}\bigg),
$$
where \red{$n_j$ is given by \eqref{1.mom}, $u_{ji}$ and $T_{ji}$ are yet to be determined.} Then the collision operator becomes
\begin{align*}
  & \widehat{Q}_{ji}(M_{ji},f_i) = \widehat{c}_{ji}\diver
	\bigg\{\widehat{A}_{ji}(v)\bigg(\na f_i
	+ m_i\frac{v-u_{ji}}{T_{ji}}f_i\bigg)\bigg\}, \\
	& \mbox{where } \widehat{A}_{ji}(v) = \int_{\R^3}A(v-v_*)M_{ji}(v_*)\mathrm{d}v_*.
\end{align*}
In this step, we used the fact $A(z)z=0$ for $z\in\R^3$ and from now on, all derivatives are with respect to $v$. Second, we suppose that the matrix $\widehat{A}_{ji}$ is independent of the velocity $v$
(otherwise, the computation of the moments becomes awkward) and that $\widehat{A}_{ji}$ is
diagonal (i.e., we neglect anisotropic diffusion). This leads to the Dougherty operator (see \cite{FJHHE22} for a similar model)
\begin{equation}\label{1.Qji}
  Q_{ji}(f_i) = c_{ji}\diver\bigg(\na f_i	+ m_i\frac{v-u_{ji}}{T_{ji}}f_i\bigg),
\end{equation}
where the coefficients $c_{ji}$ should be a reasonable approximation of the exact
expression
$$
  \widehat{c}_{ji}\widehat{A}_{ji}(v) = \frac{|\log\Lambda|q_i^2q_j^2}{8\pi\eps_0^2m_i^2}
	\int_{\R^3}A(v-v_*)M_{ji}(v_*)\mathrm{d}v_*.
$$
Assuming that the kinetic energy $m_j|v-v_*|^2$ is of the order of the thermal energy
$T_j$ (we neglected the Boltzmann constant), 
we may approximate $A(v-v_*)$ by $(T_{j}/m_j)^{(\beta+2)/2}\blue{\mathbb{I}}$, such that
we can replace $\widehat{c}_{ji}\widehat{A}_{ji}$ by 
$$
 c_{ji} := \frac{|\log\Lambda|q_i^2q_j^2}{8\pi\eps_0^2m_i^2}n_j
	\bigg(\frac{T_{j}}{m_j}\bigg)^{(\beta+2)/2},
$$
and the definition for $c_{ji}$
is exactly \eqref{1.cji} after setting $\gamma:=\beta+2$. 

To determine $u_{ji}$ and $T_{ji}$, we assume that the operator \eqref{1.Qji} conserves
the momentum and energy (mass is automatically preserved):
\begin{align}
  \int_{\R^3}Q_{ji}(f_i)m_iv\dv 
	+ \int_{\R^3}Q_{ij}(f_j)m_jv\dv &= 0, \label{1.momentum} \\
  \int_{\R^3}Q_{ji}(f_i)m_i|v|^2\dv 
	+ \int_{\R^3}Q_{ij}(f_j)m_j|v|^2\dv &= 0, \quad i,j=1,\ldots,s. \label{1.energy}
\end{align}
Then a straightforward computation leads to the expressions \eqref{1.uji} and \eqref{1.Tji}.
We summarize:

\begin{lemma}[Conservation properties]\label{lem.cons}
Let $u_{ji}$ and $T_{ji}$ be given by \eqref{1.uji} and \eqref{1.Tji}, respectively.
Then $Q_{ji}$ conserves the mass, momentum, and energy in the sense of
\eqref{1.momentum} and \eqref{1.energy}.
\end{lemma}

The collision operator $Q_{ji}$ also fulfills an H-theorem.

\begin{lemma}[Entropy decay]\label{lem.ent}
It holds formally that
$$
	\int_{\R^3}Q_{ji}(f_i)\log f_i\dv + \int_{\R^3}Q_{ij}(f_j)\log f_j\dv \le 0, \quad
	i,j=1,\ldots,s.
$$
\end{lemma}

\begin{proof}
We use definition \eqref{1.Mji} of the Maxwellian and the conservation properties of $Q_{ji}$:
\begin{align*}
  \int_{\R^3}Q_{ji}(f_i)\log M_{ij}\dv 
	&= \int_{R^3}Q_{ji}(f_i)\bigg(\log n_i+\frac32\log\frac{m_i}{2\pi T_{ji}}
	- \frac{m_i}{2T_{ji}}|v-u_{ji}|^2\bigg)\dv \\
	&= -\frac{m_i}{2T_{ji}}\int_{\R^3}Q_{ji}(f_i)|v-u_{ji}|^2\dv \\
  &= \frac{u_{ji}}{T_{ji}}\int_{\R^3}Q_{ji}(f_i)m_iv\dv - \frac{1}{2T_{ji}}\int_{\R^3}
	Q_{ji}(f_i)m_i|v|^2\dv \\
	&= -\frac{u_{ij}}{T_{ij}}\int_{\R^3}Q_{ij}(f_j)m_jv\dv + \frac{1}{2T_{ij}}\int_{\R^3}
	Q_{ij}(f_j)m_j|v|^2\dv \\
	&= \frac{m_j}{2T_{ij}}\int_{\R^3}Q_{ij}(f_j)|v-u_{ji}|^2\dv 
	= -\int_{\R^3}Q_{ij}(f_j)\log M_{ji}\dv,
\end{align*}
where we also used the symmetry of $u_{ji}$ and $T_{ji}$.
Therefore, \eqref{1.gradflow} yields
\begin{align*}
  \int_{\R^3}Q_{ji}(f_i)&\log f_i\dv + \int_{\R^3}Q_{ij}(f_j)\log f_j\dv \\
	&= \int_{\R^3}Q_{ji}(f_i)\log\frac{f_i}{M_{ij}}\dv
	+ \int_{\R^3}Q_{ij}(f_j)\log\frac{f_j}{M_{ji}}\dv \\
	&= -\int_{\R^3}c_{ji}f_i\bigg|\na\log\frac{f_i}{M_{ij}}\bigg|^2\dv
	- \int_{\R^3}c_{ij}f_j\bigg|\na\log\frac{f_j}{M_{ij}}\bigg|^2\dv \le 0,
\end{align*}
ending the proof.
\end{proof}

\begin{remark}\rm
For later use, we note that it holds formally that
\begin{align}\label{1.QlogM}
  0 &= -\frac12\sum_{i,j=1}^s\int_{\R^3}\big(Q_{ji}(f_i)\log M_{ij}\dv 
	+ Q_{ij}(f_j)\log M_{ji}\big)\dv \\
	&= -\sum_{i,j=1}^s \int_{\R^3}Q_{ji}(f_i)\log M_{ij}\dv
	= \sum_{i,j=1}^s\int_{\R^3}c_{ji}f_i\na\log\frac{f_i}{M_{ij}}\cdot\na\log M_{ji}\dv.
	\nonumber
\end{align}
\end{remark}

\blue{To summarize, in order to obtain the simplified system \eqref{1.eq}--\eqref{1.Tji} from the true multispecies Landau system, we have made two levels of approximations: 1) replacing $f_j$ in $\widehat{Q}_{ji}(f_j,f_i)$ by its corresponding Maxwellian; 2) replacing the kernel matrix $A(v-v_*)$ by a multiple of the identity matrix $(T_j/m_j)^{(\beta+2)/2}\mathbb{I}$ and matching the unit. These two approximations seem to depart significantly from the true Landau model, yet we are able to retain its most basic properties, namely, conservation of mass, momentum, and energy, as well as decay of entropy. It's worth noting that the conservation of momentum and energy only holds at the global level, that is, there is still momentum/energy exchange between species. Specifically, only $n_i$ and $\rho_i$, and global velocity $u$ and temperature $T$ ($u:=\sum_i \rho_i u_i/\sum_i \rho_i$, $T=\sum_i m_i \int f_i|v-u|^2\dv /3 \sum_i n_i$) remain as constant, but $u_i$, $T_i$, $c_{ji}$, $u_{ji}$, and $T_{ji}$ do not. Furthermore, when all species are the same, the system \eqref{1.eq}--\eqref{1.Tji} just reduces to the single-species Fokker-Planck equation 
\begin{align}
 \pa_t f = c\diver\left(\na f + m\frac{v-u}{T}f\right),
\end{align}
with $n= \int f \dv$, $u = \frac{1}{n}\int f v\dv$, $T=\frac{m}{3n}\int f|v-u|^2\dv$, and 
$c= \frac{|\log\Lambda|q^4}{8\pi\eps_0^2m^2}n
	\left(\frac{T}{m}\right)^{(\beta+2)/2}$.
This equation is linear since $n$, $u$, $T$, and $c$ all remain as constant due to conservation of mass, momentum, and energy.
}


\section{Proof of Theorem \ref{thm.ex}}\label{sec.ex}

We prove the existence of weak solutions by introducing an approximate scheme,
deriving suitable estimates uniform in the approximation parameters, and then
passing to the limit of vanishing approximation parameters. 
Recall that $\vv:=(1+|v|^2)^{1/2}$ and $g(v)=\pi^{-3/2}e^{-|v|^2}$ for $v\in\R^3$.
We set $z^+=\max\{0,z\}$ for $z\in\R$, and
we choose the parameters $p>3$ and $K>0$ sufficiently large (to be specified later). 
Our approximated system is based on three approximation levels:
the truncated domain size $M>0$, the truncation parameter $0<\eps<1$, and the regularization
parameter $0<\delta<1$:
\begin{align}\label{2.approx}
  \pa_t f_i &+ \delta\bigg(\vv^K f_i - g(v)\int_{B_M}\vv^K f_i^+ \dv
	\bigg) - \delta\diver\big(|\na f_i|^{p-2}\na f_i\big) \\
	&= \sum_{j=1}^s c_{ji}^\eps[f]\diver\bigg(\na f_i + \frac{m_if_i}{T_{ji}^\eps[f]}
	(v-u_{ji}[f])\bigg)\quad\mbox{in }B_M,\ t>0, \nonumber
\end{align}
with the initial conditions \eqref{1.ic} and the no-flux boundary conditions 
\begin{equation}\label{1.bc}
  \bigg\{\delta|\na f_i|^{p-2}\na f_i
	+ \sum_{j=1}^s c_{ji}^\eps[f]\bigg(\na f_i + \frac{m_if_i}{T_{ji}^\eps[f]}
	(v-u_{ji}[f])\bigg)\bigg\}\cdot\nu = 0\quad\mbox{on }\pa B_M,\ t>0,
\end{equation}
where $f=(f_1,\ldots,f_s)$, $B_M\subset\R^3$ is the ball around the origin with radius $M$, 
and $\nu$ is the exterior unit normal vector to $\pa B_M$. The nonlinear coefficients
are approximated by 
\begin{align}
  c_{ji}^\eps[f] &= 
  \begin{cases}\displaystyle
  \frac{|\log\Lambda|q_i^2 q_j^2}{8\pi\eps_0^2 m_i^2}n_j \bigg(
  \frac{T^{\eps,\uparrow}_j[f]}{m_j}\bigg)^{\gamma/2} + \eps, & \gamma\geq 0\\
	\displaystyle
  \frac{|\log\Lambda|q_i^2 q_j^2}{8\pi\eps_0^2 m_i^2}n_j \bigg(
  \frac{T^{\eps,\downarrow}_j[f]}{m_j}\bigg)^{\gamma/2} + \eps, & \gamma < 0
  \end{cases}
  \nonumber \\ 
	u_{ji}^\eps [f] &= \frac{c_{ji}^\eps[f] m_i\rho_i u_i^\eps[f] + c_{ij}^\eps[f] m_j\rho_j 
	u_j^\eps[f]}{c_{ji}^\eps[f] m_i\rho_i + c_{ij}^\eps[f] m_j\rho_j}, \label{2.ujieps} \\
	T^\eps_{ji}[f] &= \frac{c_{ji}^\eps[f]\rho_i T^{\eps,\downarrow}_i[f] 
	+ c_{ij}^\eps[f]\rho_j T^{\eps,\downarrow}_j[f]}{c_{ji}^\eps [f]\rho_i 
	+ c_{ij}^\eps[f]\rho_j} \nonumber \\ 
  &\phantom{xx}{}+ \frac{c_{ji}^\eps[f] m_i\rho_i c_{ij}^\eps[f] m_j\rho_j |u_i^\eps[f] 
	- u_j^\eps[f]|^2}{3(c_{ji}^\eps[f]\rho_i + c_{ij}^\eps[f]\rho_j)(c_{ji}^\eps[f] m_i\rho_i 
	+ c_{ij}^\eps[f] m_j\rho_j)}, \nonumber
\end{align}
and the (truncated) moments are defined according to
\begin{align*}
  n_i &= \int_{\R^3}f_i^0\dv, \quad \rho_i = m_in_i, \\
  u_i^\eps[f] &= \frac{1}{n_i}\int_{B_M}\min\bigg\{f_i^+,\frac{g(v)}{\eps}\bigg\}v\dv, \\
	T^{\eps,\uparrow}_j[f] &= \frac{m_i}{3n_i}\int_{B_M}\min\bigg\{f_i^+,\frac{g(v)}{\eps}\bigg\}
	|v-u_i^\eps[f]|^2\dv, \\
	T^{\eps,\downarrow}_j[f] &= \frac{m_i}{3n_i}\int_{B_M}\max\big\{f_i,\eps g(v)\big\}
	|v-u_i^\eps[f]|^2\dv.
\end{align*}
Note that $n_i$ is given by the initial datum $f_i^0$ because of mass conservation. 
The truncations guarantee that for all $f_1,\ldots,f_s\in L^1(\R^3;\vv^2\dv)$, the integrals
$u_i^\eps[f]$, $T^{\eps,\uparrow}_j[f]$, and $T^{\eps,\downarrow}_j[f]$ are well defined and
\begin{equation}\label{2.cuT}
  \eps\le c_{ji}^\eps[f]\le C(\eps), \quad |u_{ji}^\eps[f]|\le C(\eps), \quad
	c\eps\le T_{ji}^\eps[f]<\infty
\end{equation}
for some constants $c>0$ and $C(\eps)>0$ which are independent of $M$.

\subsection{Existence of solutions to the approximated system}

We show that there exists a weak solution $f_i$ to \eqref{1.ic}, \eqref{2.approx},
and \eqref{1.bc} by reformulating
the equations as a fixed-point problem for a suitable mapping. For this, we introduce
the space $X=L^p(0,T;L^p(B_M))$ recalling that $p>3$. 
Let $\sigma\in[0,1]$ and $\widehat f_i\in X$, $i=1,\ldots,s$, be given.
We consider first the partially linearized equations
\begin{align}\label{2.lin}
  \pa_t f_i	&+\delta\bigg(\vv^K f_i - \sigma g(v)\int_{B_M}\vv^K\widehat{f}_i^+ \dv\bigg)
	+ \delta|f_i|^{p-2}f_i - \delta\diver\big(|\na f_i|^{p-2}\na f_i\big) \\ 
	&{}- \sigma\sum_{j=1}^s c_{ji}^\eps[\widehat{f}]\diver\bigg(\na f_i 
	+ \frac{m_if_i}{T_{ji}^\eps[\widehat{f}]}(v-u^\eps_{ji}[\widehat{f}])\bigg)
	= \sigma\delta |\widehat{f}_i|^{p-2}\widehat{f}_i, \nonumber
\end{align}
where $i=1,\ldots,s$, with initial and no-flux boundary conditions. 
This system can be formulated as the evolution equation
$\pa_t f_i + A[\widehat{f}]f_i=b_i[\widehat{f}]$ for $t>0$, where
\begin{align*}
  A[\widehat{f}]f_i &= \delta\vv^K f_i + \delta|f_i|^{p-2}f_i 
	- \delta\diver\big(|\na f_i|^{p-2}\na f_i\big) \\ 
	&\phantom{xx}{}- \sigma\sum_{j=1}^s c_{ji}^\eps[\widehat{f}]
	\diver\bigg(\na f_i + \frac{m_if_i}{T_{ji}^\eps[\widehat{f}]}
	(v-u^\eps_{ji}[\widehat{f}])\bigg), \\
	b_i &= \sigma g(v)\int_{B_M}\vv^K\widehat{f}_i^+ \dv
	+ \sigma\delta |\widehat{f}_i|^{p-2}\widehat{f}_i.
\end{align*}
The operator $A[\widehat{f}]:V\to V'$ with $V=W^{1,p}(B_M)$ and its dual space $V'$ is monotone,
hemicontinuous, and coercive. We conclude from \cite[Theorem 30.A]{Zei90} that \eqref{2.lin}
possesses a unique solution $f_i\in L^p(0,T;V)$ with $\pa_t f_i\in L^{p/(p-1)}(0,T;V')$,
$i=1,\ldots,s$.

Next, we use the test function $f_i$ in the weak formulation of \eqref{2.lin}:
\begin{align}\label{2.aux2}
  \frac12&\int_{B_M}f_i(t)^2\dv - \frac12\int_{B_M}(f_i^0)^2\dv
	+ \delta\int_0^t\int_{B_M}|f_i|^p\dv\ds 
	+ \delta\int_0^t\int_{B_M}|\na f_i|^p\dv\ds \\
  &= -\delta\int_0^t\int_{B_M}\vv^Kf_i^2\dv\ds 
	+ \sigma\int_0^t\bigg(\int_{B_M}f_i g(v)\dv\bigg)\bigg(\int_{B_M}\vv^K\widehat{f}_i^+ 
	\dv\bigg)\ds \nonumber \\
	&\phantom{xx}{}- \sigma\sum_{j=1}^n\int_0^t\int_{B_M}c_{ji}^\eps[\widehat{f}]
	\bigg(|\na f_i|^2 + \frac{m_i f_i}{T_{ji}^\eps[\widehat{f}]}(v-u_{ji}^\eps[\widehat{f}])
	\cdot\na f_i\bigg)\dv\ds \nonumber \\
	&\phantom{xx}{}+ \sigma\delta\int_0^t\int_{B_M}|\widehat{f}_i|^{p}\dv\ds. \nonumber 
\end{align}
Taking into account that we integrate over a bounded domain, and in particular that 
$\vv^K$ is bounded, we estimate the second term on the right-hand side as follows,
using H\"older's inequality as well as the embeddings
$L^p(B_M)\hookrightarrow L^1(B_M)$ and $L^p(B_M)\hookrightarrow L^{p/(p-1)}(B_M)$:
\begin{align*}
  \sigma\int_0^t&\bigg(\int_{B_M}f_i g(v)\dv\bigg)
	\bigg(\int_{B_M}\vv^K\widehat{f}_i^+ f_i\dv\bigg)\ds
	\le \int_0^t\|f_i\|_{L^1(B_M)}\|\widehat{f}_i\|_{L^p(B_M)}
	\|f_i\|_{L^{p/(p-1)}(B_M)}\ds \\
	&\le \int_0^t\|f_i\|_{L^1(B_M)}^{p/(p-1)}\|f_i\|_{L^{p/(p-1)}(B_M)}^{p/(p-1)}\ds
	+ C\int_0^t\|\widehat{f}_i\|_{L^p(B_M)}^p\ds \\
	&\le C(M)\int_0^t\|f_i\|_{L^p(B_M)}^{2p/(p-1)}\ds 
	+ C\int_0^t\|\widehat{f}_i\|_{L^p(B_M)}^p\ds.
\end{align*}
Since $2p/(p-1)<p$ (because of $p>3$), the elementary inequality
$z^{2p/(p-1)}\le C(\delta) + (\delta/2) z^p$ for $z\ge 0$ yields
\begin{align*}
  \sigma\int_0^t&\bigg(\int_{B_M}f_i g(v)\dv\bigg)\bigg(\int_{B_M}\vv^K\widehat{f}_i^+ f_i\dv
	\bigg)\ds	\\
	&\le \int_0^t\bigg(C(\delta) + \frac{\delta}{2}\|f_i\|_{L^p(B_M)}^p	
	+ C\|\widehat{f}_i\|_{L^p(B_M)}^p\bigg)\ds,
\end{align*}
and the second term on the right-hand side can be absorbed by the left-hand side of
\eqref{2.aux2}. We write $f_i\na f_i=\frac12\na f_i^2$, use Young's inequality, 
and integrate by parts
in the third term on the right-hand side of \eqref{2.aux2}
(we denote the measure on $\pa B_M$ with $d\Sigma_v$):
\begin{align*}
  -\sigma\sum_{j=1}^n&\int_0^t\int_{B_M}c_{ji}^\eps[\widehat{f}]
	\bigg(|\na f_i|^2 + \frac{m_i f_i}{T_{ji}^\eps[\widehat{f}]}(v-u_{ji}^\eps[\widehat{f}])
	\cdot\na f_i\bigg)\dv\ds \\
	&\le -\frac{\sigma}{2}\sum_{i,j=1}^s\int_0^t\int_{B_M}
	c_{ji}^\eps[\widehat{f}]|\na f_i|^2\dv\ds
		+\frac{\sigma}{2}\sum_{i,j=1}^s \int_0^t\int_{B_M}c_{ji}^\eps[\widehat{f}]
	\frac{m_i^2|u_{ji}^\eps[\widehat{f}]|^2}{T_{ji}^\eps[\widehat{f}]^2}
	f_i^2 \dv\ds \\
	&\phantom{xx}{}-\frac{\sigma}{2}\sum_{i,j=1}^s\int_0^t\int_{\pa B_M}
	c_{ji}^\eps[\widehat{f}]\frac{m_i}{T_{ji}^\eps[\widehat{f}]}|v|f_i^2 \mathrm{d}\Sigma_v\ds 
	+ 3\sum_{i,j=1}^s\int_0^t\int_{B_M}c_{ji}^\eps[\widehat{f}]
	\frac{m_i}{T_{ji}^\eps[\widehat{f}]}f_i^2\dv\ds \\
	&\le C(\eps)\int_0^t\int_{B_M}f_i^2\dv\ds,
\end{align*}
using $v\cdot\nu=|v|$ and bounds \eqref{2.cuT} in the last step. Then \eqref{2.aux2} gives
\begin{align*}
  \int_{B_M}f_i(t)^2\dv + \delta\int_0^t\|f_i\|_{W^{1,p}(B_M)}^p\ds
	\le C(\delta) + C(\eps)\int_0^t\int_{B_M}f_i^2\dv\ds
	+ C\int_0^t\|\widehat{f}_i\|_{L^p(B_M)}^p\ds,
\end{align*}
and it follows from Gronwall's inequality that, for any $T>0$,
\begin{equation}\label{2.aux}
  \sup_{0<t<T}\|f_i\|_{L^2(B_M)}^2 + \int_0^T\|f_i\|_{W^{1,p}(B_M)}^p\mathrm{d}t
	\le C\big(\delta,\eps,T,\|f^0\|_{L^2(\R^3)}\big)
	\bigg(1+\int_0^T\|\widehat{f}_i\|_{L^p(B_M)}^p\mathrm{d}t\bigg).
\end{equation}
This estimate allows us to derive a bound for the time derivative,
\begin{align}\label{2.aux.2}
  \|\pa_t f_i\|_{L^{p/(p-1)}(0,T;W^{1,p}(B_M)')}\le C(\delta,\eps,T,f^0)
  \bigg(1+\int_0^T\|\widehat{f}_i\|_{L^p(B_M)}^p\mathrm{d}t\bigg).
\end{align}
Estimate \eqref{2.aux} shows that the mapping $F:X\times[0,1]\to X$, 
$(\widehat{f},\sigma)\mapsto f$, is well defined. 
Moreover, the function $F(\cdot,0):X\to X$ is constant. 


\textcolor{black}{
The (sequential) continuity of $F$ is shown as follows. Let $(\widehat{f}^{(n)},\sigma^{(n)})_{n\in\N}\subset X\times [0,1]$ be a sequence such that $\sigma^{(n)}\to \sigma$, $\widehat f^{(n)}\to \widehat f$ in $X$. Let $f^{(n)}=F(\widehat f^{(n)},\sigma^{(n)})$, $f = F(\widehat{f},\sigma)$. We show that $f^{(n)}\to f$ in $X$. From \eqref{2.aux}, \eqref{2.aux.2} it follows that $f^{(n)}$ is bounded in $L^p(0,T; W^{1,p}(B_M))$ and $\pa_t f^{(n)}$ is bounded in $L^{p/(p-1)}(0,T;W^{1,p}(B_M)')$, hence by Aubin-Lions Lemma it follows that $f^{(n)}$ is strongly convergent in $X = L^p(0,T; L^p(\R^3))$ up to subsequences. Taking the limit $n\to\infty$ in $\pa_t f_i^{(n)} + A[\widehat{f}^{(n)}]f_i^{(n)} = b_i[\widehat{f}^{(n)}]$ and exploiting the bounds \eqref{2.aux}, \eqref{2.aux.2} one shows that the limit $g$ of (a subsequence of) $f^{(n)}$ satisfies $\pa_t g + A[\widehat{f}]g = b_i[\widehat{f}]$. The uniqueness of the solution to \eqref{2.lin} yields $g=f$, hence the continuity of $F$ is proved.
}

The compactness of $F$ follows from the compact embedding
$W^{1,p}(B_M)\hookrightarrow L^p(B_M)$, the bounds for $f_i$ in $L^p(0,T;W^{1,p}(B_M))$
and $W^{1,p/(p-1)}(0,T;W^{1,p}(B_M)')$, and the Aubin--Lions lemma \cite{Sim87}.

To apply the Leray--Schauder fixed-point theorem, we need to show that the set
$\{f\in X:F(f,\sigma)=f\}$ of fixed points of $F(\cdot,\sigma)$ is bounded in $X$
uniformly in $\sigma\in[0,1]$. To this end, we set $\widehat{f}=f$ in \eqref{2.lin},
use the test function $f_i$ in its weak formulation, and estimate similarly as above:
\begin{align*}
  \frac12&\int_{B_M}f_i^2(t)\dv - \frac12\int_{B_M}(f_i^0)^2\dv
	+ \delta(1-\sigma)\int_0^t\int_{B_M}|f_i|^p\dv\ds
	+ \delta\int_0^t\int_{B_M}|\na f_i|^p\dv\ds \\
	&\le -\int_0^t\int_{B_M}\vv^K f_i^2\dv\ds
	+ \int_0^t\bigg(\int_{B_M}f_ig(v)\dv\bigg)\bigg(\int_{B_M}\vv^K f_i^+\dv\bigg)\ds \\
	&\phantom{xx}{}- \eps\int_0^t\int_{B_M}|\na f_i|^2\dv\ds
	+ C(\eps,M)\int_0^t\int_{B_M}f_i^2\dv\ds
	\le C(\eps,M)\int_0^t\int_{B_M}f_i^2\dv\ds,
\end{align*}
where we used the inequality $(\int_{B_M}f_i\dv)^2\le C(M)\int_{B_M}f_i^2\dv$.
We deduce from Gronwall's inequality and the Poincar\'e--Wirtinger inequality
that $f_i$ is bounded in $L^p(0,T;W^{1,p}(B_M))$
uniformly in $\sigma\in[0,1]$. Therefore, we can apply the Leray--Schauder fixed-point
theorem to infer the existence of a fixed point to \eqref{2.lin} with $\sigma=1$, i.e.\ a
solution $f_i\in L^p(0,T;L^p(B_M))$, $i=1,\ldots,s$, to \eqref{2.approx}.


\subsection{Limit $M\to\infty$}

Let $f_i^M:=f_i$ be a weak solution to \eqref{2.approx}.
We first derive some estimates uniform in $M$ and then pass to the limit $M\to\infty$.

\begin{lemma}\label{lem.mass}
The solution $f_i^M$ to \eqref{2.approx}, constructed in the previous subsection, is
nonnegative in $B_M\times(0,T)$, and the mass is controlled, $\|f_i^M(t)\|_{L^1(B_M)}
\le \|f_i^0\|_{L^1(B_M)}$ for $t>0$. 
\end{lemma}

\begin{proof}
We use the test function $(f_i^M)^-=\min\{0,f_i^M\}$ in the weak formulation
of \eqref{2.approx}, use $(f_i^0)^-=0$, and integrate by parts in the collision operator:
\begin{align*}
  \frac12&\int_{B_M}(f_i^M)^-(t)^2\dv 
	+ \delta\int_0^t\int_{B_M}|\na (f_i^M)^-|^p\dv\ds \\
	&= \sum_{j=1}^s\int_0^t\int_{B_M}c_{ji}[f^M]\bigg(
	- \frac{1}{2}|\na (f_i^M)^-|^2
	+ \frac32\frac{m_i}{T_{ji}[f^M]}|(f_i^M)^-|^2\bigg)\dv\ds \\
	&\phantom{xx}
		+\frac{1}{2}\sum_{i,j=1}^s \int_0^t\int_{B_M}c_{ji}^\eps[f^M]
		\frac{m_i^2|u_{ji}^\eps[f^M]|^2}{T_{ji}^\eps[f^M]^2}
		|(f_i^M)^-|^2 \dv\ds \\
	&\phantom{xx}
		-\frac{1}{2}\sum_{i,j=1}^s\int_0^t\int_{\pa B_M}c_{ji}^\eps[f^M]
		\frac{m_i}{T_{ji}^\eps[f^M]}|v| |(f_i^M)^-|^2 \mathrm{d}\Sigma_v\ds \\
	&\phantom{xx}{}- \delta\int_0^t\int_{B_M}\vv^K |(f_i^M)^-|^2\dv\ds \\
	&\phantom{xx}{}+ \delta\int_0^t\bigg(\int_{B_M}(f_i^M)^-g(v)\dv\bigg)
	\bigg(\int_{B_M}\vv^K(f_i^M)^+\dv\bigg)\ds \\
	&\le C(\eps)\int_0^t\int_{B_M}|(f_i^M)^-|^2\dv\ds,
\end{align*}
since the last term in the last but one step is nonpositive.
We conclude from Gronwall's lemma that $(f_i^M)^-(t)=0$ and hence $f_i^M(t)\ge 0$
in $B_M$ for $t>0$.
Next, we use the test function $\phi=1$ in the weak formulation of \eqref{2.approx}:
\begin{align*}
  \int_{B_M}f_i^M(t)\dv &= \int_{B_M}f_i^0\dv - \delta\int_0^t\int_{B_M}\vv^K f_i^M\dv\ds \\
	&\phantom{xx}{}+ \delta\bigg(\int_0^t\int_{B_M}g(v)\dv\bigg)
	\bigg(\int_{B_M}\vv^K f_i^M\dv\bigg)\ds	\le \int_{B_M}f_i^0\dv,
\end{align*}
since $\int_{B_M}g(v)\dv\le\int_{\R^3}g(v)\dv=1$. This proves the mass control.
\end{proof}

We show now some estimates uniform in $M$.

\begin{lemma}\label{lem.unifM}
Let $0<\theta<1-3/p$. Then there exists a constant $C(\delta,\eps)>0$ 
independent of $M$ such that
\begin{align*}
  \sup_{0<t<T}&\int_{B_M}\big(f_i^M(t)^2+\vv^\theta f_i^M(t)\big)\dv
	+ \int_0^T\int_{B_M}\vv^{K+\theta}f_i^M\dv\ds \\
	&{}+ \int_0^T\int_{B_M}\big(|\na f_i^M|^2 + |\na f_i^M|^p\big)\dv\ds \le C(\delta,\eps).
\end{align*}
\end{lemma}

\begin{proof}
We use the test function $f_i^M$ in the weak formulation of \eqref{2.approx}, use
$\eps\le c_{ji}[f^M]\le C(\eps)$, and integrate by parts
in the drift part of the collision operator:
\begin{align*}
  \frac12&\int_{B_M}f_i^M(t)^2\dv - \frac12\int_{B_M}(f_i^0)^2\dv
	+ \delta\int_0^t\int_{B_M}\vv^K (f_i^M)^2\dv\ds
	+ \delta\int_0^t\int_{B_M}|\na f_i^M|^p\dv\ds \\
	&\le \delta\int_0^t\bigg(\int_{B_M}f_i^M g(v)\dv\bigg)\bigg(\int_{B_M}\vv^K f_i^M\dv\bigg)\ds
	- \eps\int_0^t\int_{B_M}|\na f_i^M|^2\dv\ds \\
	&\phantom{xx}{}+ C(\eps)\int_0^t\int_{B_M}(f_i^M)^2\dv\ds.
\end{align*}
Because of the mass control from Lemma \ref{lem.mass},
$\int_{B_M}f_i^M g(v)\dv\le \int_{B_M}f_i^M\dv\le C(f_i^0)$. Hence,
\begin{align}
  \frac12&\int_{B_M}f_i^M(t)^2\dv + \delta\int_0^t\int_{B_M}\vv^K (f_i^M)^2\dv\ds
	+ \delta\int_0^t\int_{B_M}|\na f_i^M|^p\dv\ds	\label{3.aux0} \\
	&\phantom{xx}{}+ \eps\int_0^t\int_{B_M}|\na f_i^M|^2\dv\ds \nonumber \\
	&\le C + C(f_i^0)\int_0^t\int_{B_M}\vv^K f_i^M\dv\ds
	+ C(\eps)\int_0^t\int_{B_M}(f_i^M)^2\dv\ds. \nonumber
\end{align}

To control the second term on the right-hand side, 
we derive a bound for $\vv^{K+\theta} f_i^M$ for some $\theta>0$.
This is done by using the test function $\vv^\theta$ in \eqref{2.approx}:
\begin{align}\label{3.aux}
  \int_{B_M}&\vv^\theta f_i^M(t)\dv - \int_{B_M}\vv^\theta f_i^0\dv
	+ \delta\int_0^t\int_{B_M}\vv^{K+\theta} f_i^M\dv \\
	&\le C(g)\int_0^t\int_{B_M}\vv^K f_i^M\dv\ds 
	+ \delta C\int_0^t\int_{B_M}\vv^{\theta-2}|\na f_i^M|^{p-2}|\na f_i^M\cdot v|\dv\ds 
	\nonumber \\
	&\phantom{xx}{}- \theta\sum_{j=1}^s\int_0^t\int_{B_M} c_{ji}^\eps[f^M]\vv^{\theta-2}v
	\cdot\bigg(\na f_i^M + \frac{m_i f_i^M}{T_{ji}[f^M]}(v-u_{ji}[f^M])\bigg)\dv\ds \nonumber \\
	&=: I_1+I_2+I_3, \nonumber
\end{align}
where $C(g)>0$ depends on the integral $\int_{B_M}\vv^\theta g(v)\dv$ which is bounded
uniformly in $M$. The first term is estimated according to
\begin{align*}
  I_1 &\le \int_0^t\int_{B_M}\bigg(\frac{\delta}{4}\vv^{K+\theta}
	+ C(\delta,g,K)\bigg)f_i^M\dv\ds \\
	&\le \frac{\delta}{4}\int_0^t\int_{B_M}\vv^{K+\theta}f_i^M\dv\ds + C(\delta,g,K,f_i^0),
\end{align*}
and the integral on the right-hand side can be absorbed by the left-hand side of \eqref{3.aux}.
We use Young's inequality with exponents $p$ and $p/(p-1)$ to find that
\begin{align*}
  I_2 &\le \delta C\int_0^t\int_{B_M}\vv^{\theta-1}|\na f_i^\eps|^{p-1}\dv\ds \\
	&\le \frac{\delta}{2}\int_0^t\int_{B_M}|\na f_i^M|^{p}\dv\ds
	+ C\delta\int_0^t\int_{B_M}\vv^{p(\theta-1)}\dv\ds.
\end{align*}
The integral over $\vv^{p(\theta-1)}$ is bounded uniformly in $M$
if $p(\theta-1)<-3$, which is equivalent to $\theta<1-3/p$. 
We integrate by parts in the first part of $I_3$:
\begin{align*}
  -\sum_{j=1}^s\int_0^t\int_{B_M}c_{ji}^\eps[f^M]\vv^{\theta-2}v\cdot\na f_i^M\dv\ds
	&= \sum_{j=1}^s\int_0^t\int_{B_M}c_{ji}^\eps[f^M]\diver(\vv^{\theta-2}v)f_i^M\dv\ds \\
	&\phantom{xx}{}
	- \sum_{j=1}^s\int_0^t\int_{\pa B_M}c_{ji}^\eps[f^M]\vv^{\theta-2}(v\cdot\nu)f_i^M\dv\ds,
\end{align*}
where $\nu$ is the exterior unit normal vector to $\pa B_M$. Since $B_M$ is a ball around
the origin, $\nu=v/|v|$ and hence $v\cdot\nu=|v|$, and 
we infer that the surface integral is nonpositive. 
Then, using $\vv^{\theta-2}\le 1$ and the mass control,
\begin{align*}
  -\theta\sum_{j=1}^s\int_0^t\int_{B_M}c_{ji}^\eps[f^M]\vv^{\theta-2}v\cdot\na f_i^M\dv\ds
	&\le C(\eps)\sum_{j=1}^s\int_0^t\int_{B_M}\vv^{\theta-2}f_i^M\dv\ds 
	\le C(\eps,f_i^0).
\end{align*}
The second part of $I_3$ is estimated according to
\begin{align*}
  \theta&\sum_{j=1}^s\int_0^t\int_{B_M}c_{ji}^\eps[f^M]\frac{m_i}{T_{ji}[f^M]}
	\vv^{\theta-2}\big(|v|^2-v\cdot u_{ji}^\eps[f^M]\big)f_i^M\dv\ds \\
	&\le C(\eps)\int_0^t\int_{B_M}(\vv^\theta + \vv^{\theta-1})f_i^M\dv\ds 
	\le C(\delta,\eps) + \frac{\delta}{4}\int_0^t\int_{B_M}\vv^{K+\theta}f_i^M\dv\ds.
\end{align*}

Summarizing, we infer from \eqref{3.aux} that 
$$
   \int_{B_M}\vv^\theta f_i^M(t)\dv 
	+ \frac{\delta}{2}\int_0^t\int_{B_M}\vv^{K+\theta} f_i^M\dv
	\le C(\delta,\eps) + \frac{\delta}{2}\int_0^t\int_{B_M}|\na f_i^M|^{p}\dv\ds.
$$
We add this inequality to \eqref{3.aux0} and use the inequality $\vv^K\le C(\delta) 
+ (\delta/8)\vv^{K+\theta}$ as well as the mass control:
\begin{align*}
  \int_{B_M}\bigg(\frac12 f_i^M(t)^2 + \vv^\theta f_i^M(t)\bigg)\dv 
	&+ \int_0^t\int_{B_M}\bigg(\frac{\delta}{2}\vv^{K+\theta}f_i^M + \eps|\na f_i^M|^2 
	+ \frac{\delta}{2}|\na f_i^M|^p\bigg)\dv\ds \\
	&\le C(\delta,\eps) + C(\eps)\int_0^t\int_{B_M}(f_i^M)^2\dv\ds.
\end{align*}
We apply Gronwall's lemma and then take the
supremum over $t\in(0,T)$ to finish the proof.
\end{proof}

Lemma \ref{lem.unifM} gives uniform bounds for $f_i^M$ in 
$L^\infty(0,T;L^2(B_M))$ and $L^p(0,T;W^{1,p}(B_M))$. Then, together with the
bounds \eqref{2.cuT}, we infer that
$\pa_t f_i^M$ is bounded in $L^{p/(p-1)}(0,T;$ $W^{-1,p}(B_M)')$ uniformly in $M$. The condition
$p>3$ implies that the embedding $W^{1,p}(B_M)$ $\hookrightarrow L^\infty(B_M)$ is
compact. Then the Aubin--Lions lemma, together with a Cantor diagonal argument,
yields the existence of a subsequence, which is not relabeled, such that, as
$M\to\infty$,
$$
  f_i^M\to f_i\quad\mbox{strongly in }L^p(0,T;L^\infty(B))
	\quad\mbox{for every ball }B\subset\R^3.
$$

We claim that
$$
  f_i^M\to f_i\quad\mbox{strongly in }L^1(0,T;L^1(\R^3)).
$$
Indeed, we know from Lemma \ref{lem.unifM} that  
$\int_{B}\vv^\theta f_i^M(t)\dv$ $\le C$ for all balls $B\subset\R^3$ 
uniformly in $M$ and for $t\in(0,T)$. Then Fatou's lemma implies that 
$$
  \int_{\R^3}\vv^\theta f_i(t)\dv
	= \int_{\R^3}\liminf_{M\to\infty}\vv^\theta f_i^M(t)\mathrm{1}_{B_M}\dv
	\le \liminf_{M\to\infty}\int_{\R^3}\vv^\theta f_i^M(t)\mathrm{1}_{B_M}\dv
	\le C,
$$
and this bound holds uniformly for $t\in (0,T)$.
Set $f_i^M(t):=0$ outside of $B_M$ and let $R<M$. We write
\begin{align*}
  \int_0^T\int_{\R^3}|f_i^M-f_i|\dv\ds 
	&= \int_0^T\int_{B_R}|f_i^M-f_i|\dv\ds + \int_0^T\int_{\{R\le|v|\le M\}}|f_i^M-f_i|\dv\ds \\
	&\phantom{xx}{}+ \int_0^T\int_{\{|v|>M\}}|f_i^M-f_i|\dv\ds =: J_1^M+J_2^M+J_3^M.
\end{align*}
Because of the strong convergence of $(f_i^M)$ in $B_R$, 
we have $J_1^M\to 0$ as $M\to\infty$. We deduce from the uniform bound for 
$\vv^\theta f_i^M$ in $L^1(\R^3)$ that
$$
  J_2^M \le \frac{1}{R^\theta}\int_0^T\int_{\{R\le|v|\le M\}}\vv^\theta|f_i^M-f_i|\dv\ds
	\le \frac{C}{R^\theta}.
$$
In a similar way, since $f_i^M=0$ in $\{|v|>M\}$, we have
$$
  J_3^M \le \frac{1}{R^\theta}\int_0^T\int_{\{|v|>M\}}\vv^\theta f_i\dv\le \frac{C}{R^\theta}.
$$
We conclude that 
$$
  \limsup_{M\to\infty}\int_0^T\int_{\R^3}|f_i^M-f_i|\dv\ds \le \frac{C}{R^\theta}
	\quad\mbox{for all }R>0.
$$
Since the left-hand side is independent of $R$, it follows that
$\limsup_{M\to\infty}\int_0^T\int_{\R^3}|f_i^M - f_i|\dv\ds = 0$, proving the claim.

We also obtain, for a subsequence, the weak convergences
\begin{align*}
  \na f_i^M\rightharpoonup \na f_i &\quad\mbox{weakly in }L^p(0,T;L^p(B)), \\
	\pa_t f_i^M\rightharpoonup \pa_t f_i &\quad\mbox{weakly in }L^p(0,T;W^{1,p}(B)')
\end{align*}
as $M\to\infty$ for any ball $B\subset\R^3$. These convergences are sufficient to
pass to the limit $M\to\infty$ in \eqref{2.approx}, and the limit $f_i^\eps:=f_i$ is a weak
solution to
\begin{align}\label{3.approx}
  \pa_t f_i^\eps &+ \delta\bigg(\vv^K f_i^\eps - g(v)\int_{\R^3}\vv^K f_i^\eps \dv
	\bigg) - \delta\diver\big(|\na f_i^\eps|^{p-2}\na f_i^\eps\big) \\
	&= \sum_{j=1}^s c_{ji}^\eps[f^\eps]\diver\bigg(\na f_i^\eps 
	+ \frac{m_if_i^\eps}{T_{ji}^\eps[f^\eps]}
	(v-u_{ji}^\eps[f^\eps])\bigg)\quad\mbox{in }\R^3,\ t>0, \nonumber
\end{align}
with the initial conditions \eqref{1.ic}.


\subsection{Limit $\eps\to 0$}

Let $f_i^\eps$ be a weak solution to \eqref{1.ic} and \eqref{3.approx}. An integration yields
the conservation of mass:
\begin{equation}\label{4.mass}
  \int_{\R^3}f_i^\eps(t)\dv = n_i = \int_{\R^3}f_i^0\dv > 0.
\end{equation}
Strictly speaking, we cannot use the test function $\phi=1$ in \eqref{3.approx}
and we need to work with a cutoff function $\psi_R$; we refer to Appendix \ref{sec.app}
for details.

\begin{lemma}\label{lem.unifeps}
There exists a constant $C(\delta,T)>0$ independent of $\eps$ such that
\begin{align*}
  &\sup_{0<t<T}\sum_{i=1}^s\int_{\R^3}\big(f_i^\eps(t)^2+\vv^\theta f_i^\eps(t)\big)\dv 
	+ \sum_{i,j=1}^s\int_0^t\int_{\R^3}c_{ij}[f^\eps]|\na f_i^\eps|^2\dv\ds \\
	&\phantom{xx}{}+ \sum_{i=1}^s\int_0^T\int_{\R^3}|\na f_i^\eps|^p \dv\ds
	+ \sum_{i=1}^s\int_0^T\int_{\R^3}\big(\vv^K(f_i^\eps)^2 + \vv^{K+\theta}f_i^\eps\big)
	\dv\ds \le C(\delta,T).
\end{align*}
\end{lemma}

\begin{proof} 
We split the proof in several steps.

{\em Step 1: Test function $\vv^\theta$.}
Let $0<\theta<1-3/p$. We use $\vv^\theta$ as a test function in \eqref{3.approx}.
Again, $\vv^\theta$ cannot be used as a test function but we may use
$\vv^\theta\psi_R(v)$ for some cutoff function $\psi_R$; 
see Appendix \ref{sec.app}. Then, summing over $i=1,\ldots,s$,
\begin{align}
  \sum_{i=1}^s&\int_{\R^3}\vv^\theta f_i^\eps(t)\dv 
	- \sum_{i=1}^s\int_{\R^3}\vv^\theta f_i^0\dv
	+ \delta\int_0^t\int_{\R^3}\vv^{K+\theta}f_i^\eps\dv\ds \label{4.aux} \\
	&= \delta\int_0^t\bigg(\int_{\R^3}\vv^\theta g(v)\dv\bigg)
	\bigg(\sum_{i=1}^s\int_{\R^3}\vv^K f_i^\eps\dv\bigg)\ds \nonumber \\
	&\phantom{xx}{}
	- \delta\sum_{i=1}^s\int_0^t\int_{\R^3}|\na f_i^\eps|^{p-2}\na f_i^\eps\cdot\na\vv^\theta\dv 
	+ \sum_{i,j=1}^s\int_0^t\int_{\R^3} c_{ji}^\eps[f^\eps]\na\vv^\theta
	\cdot\na f_i^\eps\dv\ds	\nonumber \\
	&\phantom{xx}{}- \sum_{i,j=1}^s\int_0^t\int_{\R^3} c_{ji}^\eps[f^\eps]
	\frac{m_i}{T_{ji}^\eps[f^\eps]}(v-u_{ji}^\eps[f^\eps])\cdot\na\vv^\theta f_i^\eps\dv\ds 
	\nonumber \\
	&=: I_4+\cdots+I_7. \nonumber
\end{align}
We estimate the right-hand side term by term. First, the integral over $\vv^\theta g(v)$ is
bounded. Using $\vv^K\le (\delta/8)\vv^{K+\theta}+C(\delta)$ and mass conservation
\eqref{4.mass}, we can estimate
$$
  I_4 \le C(\delta) + \frac{\delta}{8}\int_0^t\int_{\R^3}\vv^{K+\theta}f_i^\eps\dv\ds,
$$
and the last integral can be absorbed by the left-hand side of \eqref{4.aux}.
Because of $|\na\vv^\theta|\le\theta\vv^{\theta-1}$ and Young's inequality, 
the term $I_5$ becomes
\begin{align*}
  I_5 &\le \delta C\int_0^t\int_{\R^3}\vv^{\theta-1}|\na f_i^\eps|^{p-1}\dv \\
	&\le \frac{C}{p}\int_0^t\int_{\R^3}\vv^{p(\theta-1)}\dv\ds
	+ \frac{p-1}{p}\delta^{p/(p-1)}\int_0^t\int_{\R^3}|\na f_i^\eps|^{p}\dv\ds \\
	&\le C + \delta^{p/(p-1)}\int_0^t\int_{\R^3}|\na f_i^\eps|^{p}\dv\ds
	\le C + \frac{\delta}{2}\int_0^t\int_{\R^3}|\na f_i^\eps|^{p}\dv\ds,
\end{align*}
taking into account that the integral over $\vv^{p(\theta-1)}$ is bounded since 
$p(\theta-1)<-3$ and choosing $\delta>0$ sufficiently small such that
$\delta^{p/(p-1)}\le\delta/2$. Integrating by parts in $I_6$ leads to
\begin{align}\label{I6.est.ulm}
  I_6 &= -\sum_{i,j=1}^s\int_0^t\int_{\R^3} c_{ji}^\eps[f^\eps]\Delta\vv^\theta f_i\dv\ds \\ \nonumber
	&\le C\sum_{i,j=1}^s\int_0^t\int_{\R^3} c_{ji}^\eps[f^\eps]\vv^{\theta-2}f_i\dv\ds
	\le C\sum_{i,j=1}^s\int_0^t c_{ji}^\eps[f^\eps]\ds,
\end{align}
where we used $\vv^{\theta-2}\le 1$ (note that $\theta<1$) and mass conservation.
It follows from Jensen's inequality, applied to the probability measure $(f_i/n_i)\dv$, that
for $q\ge 0$ and $r\ge 1$,
\begin{equation}\label{4.vv}
  \bigg(\int_{\R^3}\vv^q \frac{f_i^\eps}{n_i}\dv\bigg)^r 
	\le \int_{\R^3}\vv^{qr}\frac{f_i^\eps}{n_i}\dv.
\end{equation}
The final term $I_7$ becomes
\begin{align*}
	I_7 &= -\theta\sum_{i,j=1}^s\int_0^t\int_{\R^3}\frac{c_{ji}^\eps[f^\eps]}{T_{ji}^\eps[f^\eps]}
	m_i\vv^{\theta-2}\big(|v|^2	- v\cdot u_{ji}^\eps[f^\eps]\big)f_i^\eps\dv\ds \\
	&\le C\sum_{i,j=1}^s\int_0^t\int_{\R^3}c_{ji}^\eps[f^\eps]
	\vv^{\theta-1}\frac{|u_{ji}^\eps[f^\eps]|}{T_{ji}^\eps[f^\eps]}f_i^\eps\dv\ds \\
	&\le C\sum_{i,j=1}^s\int_0^t c_{ji}^\eps[f^\eps]
	\frac{|u_{ji}^\eps[f^\eps]|}{T_{ji}^\eps[f^\eps]}\ds,
\end{align*}
where we used $\vv^{\theta-1}\le 1$ and mass conservation. In view of 
definition \eqref{2.ujieps} and Jensen's inequality \eqref{4.vv}, we have
\begin{align}\label{4.ujieps}
	|u_{ji}^\eps[f^\eps]|^K &\le \max\big\{|u_{i}^\eps[f^\eps]|,|u_{j}^\eps[f^\eps]|\big\}^K
	\le \bigg(\sum_{i=1}^s\frac{1}{n_i}\int_{\R^3}\vv\min\{f_i^\eps,g(v)/\eps\}\dv\bigg)^K \\
	&\le C\bigg(\sum_{i=1}^s\int_{\R^3}\vv f_i^\eps\dv\bigg)^K
	\le C\sum_{i=1}^s\int_{\R^3}\vv^K f_i^\eps\dv. \nonumber
\end{align}
Thus, by Young's inequality and $\vv^K\le C(\delta)+(\delta/8)\vv^{K+\theta}$,
\begin{align}\label{4.I7}
	I_7 &\le \sum_{i,j=1}^s\int_0^t|u_{ji}^\eps[f^\eps]|^K\ds
	+ C\sum_{i,j=1}^s \int_0^t\bigg|\frac{c_{ji}^\eps[f^\eps]}{T_{ji}^\eps[f^\eps]}
	\bigg|^{K/(K-1)}\ds \\
	&\le C(\delta) 
	+ \frac{\delta}{8}\sum_{i=1}^s\int_0^t\int_{\R^3}\vv^{K+\theta} f_i^\eps\dv\ds
	+ C\sum_{i,j=1}^s \int_0^t\bigg|\frac{c_{ji}^\eps[f^\eps]}{T_{ji}^\eps[f^\eps]}
	\bigg|^{K/(K-1)}\ds. \nonumber
\end{align}
Let us distinguish two cases, according to the value of $\gamma$.

\subsection*{Case 1: $\gamma\geq 0$} 

We distinguish the subcases $\gamma\ge 2$ and $0\le\gamma<2$. First, let $\gamma\ge 2$.
Jensen's inequality \eqref{4.vv} leads to
\begin{equation*}
  c_{ji}^\eps[f^\eps]\le \eps+C|T_j^{\eps,\uparrow}|^{\gamma/2}
	\le 1+C\bigg(\int_{\R^3}\vv^2 f_i^\eps\dv\bigg)^{\gamma/2}
	\le 1+C\int_{\R^3}\vv^{ \gamma }f_i^\eps\dv.
\end{equation*}
If $0\le\gamma<2$, we apply Young's inequality:
\begin{equation*}
	c_{ji}^\eps[f^\eps]\le \eps+C|T_j^{\eps,\uparrow}|^{\gamma/2}
	\le 1+C\bigg(\int_{\R^3}\vv^2 f_i^\eps\dv\bigg)^{\gamma/2}
	\le 1+C\int_{\R^3}\vv^{ 2 }f_i^\eps\dv.
\end{equation*}
Summarizing, we obtain for all $\gamma\ge 0$:
\begin{equation}\label{4.cjieps}
	c_{ji}^\eps[f^\eps]\le 1+C\int_{\R^3}\vv^{\max\{\gamma,2\}}f_i^\eps\dv.
\end{equation}
Consequently, if we choose $K$ sufficiently large, \eqref{I6.est.ulm} yields
$$
  I_6 \le C + C\sum_{i=1}^s\int_0^t\int_{\R^3}\vv^{\max\{\gamma,2\}}f_i^\eps\dv
	\le C(\delta) + \frac{\delta}{8}\sum_{i=1}^s\int_0^t\int_{\R^3}\vv^{K+\theta}f_i^\eps\dv.
$$

To estimate the last term in \eqref{4.I7}, we bound $T_{ji}^\eps[f^\eps]$ from below. 
For this, we choose an arbitrary $\lambda>0$ and set $u_i^\eps=u_i^\eps[f^\eps]$:
\begin{align}\label{4.Tieps}
  T_{i}^{\eps,\downarrow}[f^\eps] &\ge C\int_{\R^3}f_i^\eps|v-u_i^\eps|^2\dv
	\ge C\int_{\{|v-u_i^\eps|>\lambda\}}f_i^\eps|v-u_i^\eps|^2\dv \\
	&\ge C\lambda^2\int_{\{|v-u_i^\eps|>\lambda\}}f_i^\eps\dv
	= C\lambda^2\bigg(n_i - \int_{\{|v-u_i^\eps|\le\lambda\}}f_i^\eps\dv\bigg). \nonumber
\end{align}
Applying the Cauchy--Schwarz inequality to the last integral, we have
$$
  T_{i}^{\eps,\downarrow}[f^\eps] \ge C\lambda^2\bigg\{n_i
	- \|f_i^\eps\|_{L^2(\R^3)}\bigg(\int_{\{|v-u_i^\eps|\le\lambda\}}\dv\bigg)^{1/2}\bigg\}
	\ge C\lambda^2\big(n_i - C\lambda^{3/2}\|f_i^\eps\|_{L^2(\R^3)}\big),
$$
since the integral over any ball in $\R^3$ with radius $\lambda$ is of the
order $\lambda^3$. We obtain with the choice 
$\lambda=C_0 n_i^{2/3}\|f_i^\eps\|_{L^2(\R^3)}^{-2/3}$ for some $C_0>0$:
\begin{equation*}
  T_{i}^{\eps,\downarrow}[f^\eps]\ge CC_0^2(1-CC_0^{3/2})
  n_i^{7/3}\|f_i^\eps\|_{L^2(\R^3)}^{-4/3}	
\end{equation*}
and therefore, choosing $C_0>0$ sufficiently small,
\begin{equation}\label{4.lowerT}
  T_{ji}^\eps[f^\eps] \ge \min\big\{T_i^{\eps,\downarrow}[f^\eps],
	T_j^{\eps,\downarrow}[f^\eps]\big\}
	\ge C\bigg(\sum_{k=1}^s\|f_k^\eps\|_{L^2(\R^3)}^{2}\bigg)^{-2/3}.
\end{equation}

We continue with the estimate of the last term in \eqref{4.I7}. We infer from Young's inequality
with exponents $3(K-1)/(2K)$ and $3(K-1)/(K-3)$ as well as estimate \eqref{4.cjieps} and
Jensen's inequality \eqref{4.vv} that
\begin{align*}
  \sum_{i,j=1}^s\bigg|\frac{c_{ji}^\eps[f^\eps]}{T_{ji}^\eps[f^\eps]}
	\bigg|^{K/(K-1)} &\le \sum_{i,j=1}^s T_{ji}^\eps[f^\eps]^{-3/2}
	+ C\sum_{i,j=1}^s c_{ji}^\eps[f^\eps]^{3K/(K-3)} \\
	&\le C + C\sum_{k=1}^s\|f_k^\eps\|_{L^2(\R^3)}^{2}
	+ C\sum_{i=1}^s\int_{\R^3}\vv^{3K\max\{\gamma,2\}/(K-3)}f_i^\eps\dv.
\end{align*}
For sufficiently large $K>0$, we have $3K\max\{\gamma,2\}/(K-3)<K+\theta$. Hence,
$$
  \sum_{i,j=1}^s\int_0^t\bigg|\frac{c_{ji}^\eps[f^\eps]}{T_{ji}^\eps[f^\eps]}
	\bigg|^{K/(K-1)}\ds
	\le C(\delta) + C\sum_{i=1}^s\int_0^t\|f_i^\eps\|_{L^2(\R^3)}^{2}\ds
	+ \frac{\delta}{8}\sum_{i=1}^s\int_0^t\int_{\R^3}\vv^{K+\theta}f_i^\eps\dv.
$$ 
We infer from \eqref{4.I7} that
$$
  I_7 \le C(\delta) + \frac{\delta}{4}\sum_{i=1}^s\int_0^t\int_{\R^3}\vv^{K+\theta}f_i^\eps\dv
	+ C\sum_{i=1}^s\int_0^t\int_{\R^3}(f_i^\eps)^2\dv\ds.
$$

\subsection*{Case 2: $\gamma < 0$.} 

It follows from \eqref{4.lowerT} that
\begin{align}\label{est.c.ulm}
  c_{ji}^\eps[f^\eps]\le \eps+C|T_j^{\eps,\downarrow}|^{\gamma/2}
	\leq 1 + C\|f_i^\eps\|_{L^2(\R^3)}^{-2\gamma/3}
	\leq 1 + C\bigg(\sum_{k=1}^s \|f_k^\eps\|_{L^2(\R^3)}^{2}\bigg)^{-\gamma/3}.
\end{align}
Therefore, estimates \eqref{I6.est.ulm}, \eqref{4.I7} lead to
\begin{align*}
  I_6 + I_7 &\leq C\sum_{i,j=1}^s\int_0^t c_{ji}^\eps[f^\eps]\ds + C(\delta) 
  + \frac{\delta}{8}\sum_{i=1}^s\int_0^t\int_{\R^3}\vv^{K+\theta} f_i^\eps\dv\ds \\
	&\phantom{xx}
  + C\sum_{i,j=1}^s \int_0^t\bigg|\frac{c_{ji}^\eps[f^\eps]}{T_{ji}^\eps[f^\eps]}
  \bigg|^{K/(K-1)}\ds\\
  &\leq C(\delta) + \frac{\delta}{8}\sum_{i=1}^s\int_0^t\int_{\R^3}\vv^{K+\theta} 
	f_i^\eps\dv\ds \\
	&\phantom{xx}+ C\sum_{i,j=1}^s\int_0^t \bigg(\sum_{k=1}^s \|f_k^\eps\|_{L^2(\R^3)}^{2}
  \bigg)^{K(2-\gamma)/(3(K-1))}\ds.
\end{align*}
The Gagliardo--Nirenberg inequality 
\begin{align*}
  \|f_k\|_{L^2(\R^3)}\leq C\|f_k\|_{L^1(\R^3)}^{1-\xi}\|\nabla f_k\|_{L^p(\R^3)}^\xi,
	\quad\mbox{where } \xi = \frac{3p}{2(4p-3)},
\end{align*}
and mass conservation imply that
\begin{align*}
  I_6 + I_7 &\leq C(\delta) + \frac{\delta}{8}\sum_{i=1}^s\int_0^t\int_{\R^3}\vv^{K+\theta} 
	f_i^\eps\dv\ds + \frac{\delta}{2}\int_0^t\int_{\R^3}|\na f_i^\eps|^p\dv\ds, 
\end{align*}
as long as $2\xi(2-\gamma)/3p$ or equivalently $p>(5-\gamma)/4$.

In both cases, summarizing the estimates for $I_4,\ldots,I_7$, we conclude from 
\eqref{4.aux} that
\begin{align}\label{4.aux1}
  \sum_{i=1}^s&\int_{\R^3}\vv^\theta f_i^\eps(t)\dv
	+ \frac{\delta}{2}\int_0^t\int_{\R^3}\vv^{K+\theta}f_i^\eps\dv\ds \\
	&\le C(\delta) + \frac{\delta}{2}\int_0^t\int_{\R^3}|\na f_i^\eps|^p\dv\ds 
	+ C\sum_{i=1}^n\int_0^t\int_{\R^3}(f_i^\eps)^2\dv\ds. \nonumber
\end{align}

We still need to control the integrals on the right-hand side of \eqref{4.aux1}, 
which is done in the next step.

{\em Step 2: Test function $f_i^\eps$.} We use the test function $f_i^\eps$ in \eqref{3.approx}
and sum over $i=1,\ldots,s$:
\begin{align}\label{4.aux2}
  \frac12\sum_{i=1}^s&\int_{\R^3}f_i^\eps(t)^2\dv - \frac12\sum_{i=1}^s\int_{\R^3}(f_i^0)^2\dv
	+ \delta\sum_{i=1}^s\int_0^t\int_{\R^3}\vv^K(f_i^\eps)^2\dv\ds \\
	&\phantom{xx}{}+ \delta\sum_{i=1}^s\int_0^t\int_{\R^3}|\na f_i^\eps|^p\dv\ds
	+ \sum_{i,j=1}^s\int_0^t\int_{\R^3}c_{ji}^\eps[f^\eps]|\na f_i^\eps|^2\dv\ds \nonumber \\
	&= \delta\sum_{i=1}^s\int_0^t\bigg(\int_{\R^3}f_i^\eps g(v)\dv\bigg)
	\bigg(\int_{\R^3}\vv^K f_i^\eps\dv\bigg)\ds \nonumber \\
	&\phantom{xx}{}
	- \frac12\sum_{i,j=1}^s\int_0^t\int_{\R^3}c_{ji}^\eps[f^\eps]\frac{m_i}{T_{ji}^\eps[f^\eps]}
	(v-u_{ji}^\eps[f^\eps])\cdot\na(f_i^\eps)^2\dv\ds \nonumber \\
	&=: I_8+I_9. \nonumber 
\end{align}
We use mass conservation to infer that $\int_{\R^3}f_i^\eps g(v)\dv\le \int_{\R^3}f_i^\eps\dv
\le C$ and hence,
\begin{align*}
  I_8 \le \delta C\sum_{i=1}^s\int_0^t\int_{\R^3}\vv^K f_i^\eps\dv
	\le C + \frac{\delta}{8}\sum_{i=1}^s\int_0^t\int_{\R^3}\vv^{K+\theta}f_i^\eps\dv\ds,
\end{align*}
and the last integral can be absorbed by the left-hand side of \eqref{4.aux2}.
By integration by parts and the lower bound \eqref{4.lowerT}, we have
\begin{align}\label{I9.est.ulm.1}
  I_9 &= \frac12\sum_{i,j=1}^s\int_0^t\int_{\R^3}c_{ji}^\eps[f^\eps]
	\frac{m_i}{T_{ji}^\eps[f^\eps]}\diver(v-u_{ji}^\eps[f^\eps])(f_i^\eps)^2\dv\ds \\ \nonumber
	&= \frac32\sum_{i,j=1}^s\int_0^tc_{ji}^\eps[f^\eps]
	\frac{m_i}{T_{ji}^\eps[f^\eps]}\|f_i^\eps\|_{L^2(\R^3)}^2\ds
	\le C\sum_{i,j,k=1}^s\int_0^t c_{ji}^\eps[f^\eps]\|f_{k}^\eps\|_{L^2(\R^3)}^{10/3}\ds .
\end{align}

Let $\gamma\ge 0$. Then the Gagliardo--Nirenberg inequality
with $\zeta=3p/(8p-6)\in(0,1)$ and mass conservation lead to
\begin{align*}
  I_9 \leq C\sum_{i,j,k=1}^s\int_0^t c_{ji}^\eps[f^\eps]
	\|\na f_{k}^\eps\|_{L^{p}(\R^3)}^{10\zeta/3}
  \|f_{k}^\eps\|_{L^1(\R^3)}^{10(1-\zeta)/3}\ds
  \le C\sum_{i,j,k=1}^s\int_0^t c_{ji}^\eps[f^\eps]
  \|\na f_{k}^\eps\|_{L^{p}(\R^3)}^{5p/(4p-3)}\ds.
\end{align*}
Then, using Young's inequality, estimate \eqref{4.cjieps} for $c_{ji}^\eps[f^\eps]$,
and Jensen's inequality \eqref{4.vv},
\begin{align*}
  I_9 &\le \frac{\delta}{8}\sum_{i=1}^s\int_0^t\int_{\R^3}|\na f_i^\eps|^p\dv\ds
	+ C(\delta)\sum_{i,j=1}^s|c_{ji}^\eps[f^\eps]|^{(4p-3)/(4p-8)} \\
	&\le C + \frac{\delta}{8}\sum_{i=1}^s\int_0^t\int_{\R^3}|\na f_i^\eps|^p\dv\ds
	+ C(\delta)\sum_{i=1}^s\int_0^t\int_{\R^3}\vv^{(2+\gamma)(4p-3)/(4p-8)}f_i^\eps\dv \\
	&\le C(\delta) + \frac{\delta}{8}\sum_{i=1}^s\int_0^t\int_{\R^3}|\na f_i^\eps|^p\dv\ds
	+ \frac{\delta}{8}\sum_{i=1}^s\int_0^t\int_{\R^3}\vv^{K+\theta}f_i^\eps\dv,
\end{align*}
if we choose $K+\theta>(2+\gamma)(4p-3)/(4p-8)$.

If $\gamma < 0$, estimates \eqref{est.c.ulm} and \eqref{I9.est.ulm.1} imply that
\begin{align*}
  I_9\leq C\sum_{k=1}^s\int_0^t \|f_{k}^\eps\|_{L^{2}(\R^3)}^{(10-2\gamma)/3}\ds,
\end{align*}
and Gagliardo--Nirenberg and Young's inequalities allow us to bound $I_9$ 
similarly as above as
\begin{align*}
  I_9\leq C(\delta) + \delta \sum_{i=1}^s\int_0^t\int_{\R^3}|\na f_i^\eps|^p\dv\ds,
\end{align*}
as long as $p > 2 - \gamma/4$.

In both cases, we insert the estimates for $I_8$ and $I_9$ into \eqref{4.aux2} to obtain
\begin{align*}
  \frac12\sum_{i=1}^s&\int_{\R^3}f_i^\eps(t)^2\dv	
	+ \frac{\delta}{2}\sum_{i=1}^s\int_0^t\int_{\R^3}\vv^K(f_i^\eps)^2\dv\ds 
	+ \frac{\delta}{4}\sum_{i=1}^s\int_0^t\int_{\R^3}|\na f_i^\eps|^p\dv\ds \\
	&{}+ \sum_{i,j=1}^s\int_0^t\int_{\R^3}c_{ji}^\eps[f^\eps]|\na f_i^\eps|^2\dv\ds
	\le C(\delta) + \frac{\delta}{4}\sum_{i=1}^s\int_0^t\int_{\R^3}\vv^{K+\theta}f_i^\eps\dv\ds.
\end{align*}

{\em Step 3: End of the proof.}
We add the previous inequality to \eqref{4.aux1},
\begin{align*}
  &\sum_{i=1}^s\int_{\R^3}\big(f_i^\eps(t)^2 + \vv^\theta f_i^\eps(t)\big)\dv
	+ \frac{\delta}{2}\sum_{i=1}^s\int_0^t\int_{\R^3}\vv^K(f_i^\eps)^2\dv\ds \\
	&\phantom{xx}{}+ \frac{\delta}{4}\sum_{i=1}^s\int_0^t\int_{\R^3}|\na f_i^\eps|^p\dv\ds
	+ \sum_{i,j=1}^s\int_0^t\int_{\R^3}c_{ji}^\eps[f^\eps]|\na f_i^\eps|^2\dv\ds \\
	&\phantom{xx}{}+ \frac{\delta}{4}\sum_{i=1}^s\int_0^t\int_{\R^3}\vv^{K+\theta}f_i^\eps\dv\ds
	\le C(\delta) + C\sum_{i=1}^s\int_0^t\int_{\R^3}(f_i^\eps)^2 \dv\ds.
\end{align*}
Then Gronwall's lemma concludes the proof.
\end{proof}

\begin{lemma}\label{lem.unifeps2}
There exists a constant $C(\delta,T)>0$ independent of $\eps$ and a number $r>1$ such that
$$
  \|\pa_t f_i^\eps\|_{L^r(0,T;W^{-1,p}(\R^3))}\le C(\delta,T).
$$
\end{lemma}

\begin{proof}
The estimate for $\vv^{K+\theta} f_i^\eps$ in Lemma \ref{lem.unifeps}
and bounds \eqref{4.cjieps}, \eqref{est.c.ulm} show that $c_{ji}^\eps[f^\eps]$ is uniformly bounded
in $L^{(K+\theta)/(2+\gamma)}(0,T)$ (or better), while $T_{ji}^\eps[f^\eps]^{-1}$ is uniformly 
bounded in $L^\infty(0,T)$ because of the lower bound \eqref{4.lowerT} and
the estimate for $f_i^\eps$ in $L^\infty(0,T;L^2(\R^3))$. 
Furthermore, we conclude from \eqref{4.ujieps} that 
$|u_{ji}^\eps[f^\eps]|^{K+\theta}\le C\sum_{i=1}^s\vv^{K+\theta} f_i^\eps\dv$ 
(using the Jensen inequality \eqref{4.vv}) is uniformly bounded
in $L^1(0,T)$. This shows that $c_{ji}[f^\eps]T_{ji}[f^\eps]^{-1}u_{ji}[f^\eps]$
is uniformly bounded in $L^{(K+\theta)/(3+\gamma)}(0,T)$.
Furthermore, by Young's inequality and Lemma \ref{lem.unifeps},
\begin{align*}
  \int_0^T\int_{\R^3}&(\vv^K f_i^\eps)^{(K+2\theta)/(K+\theta)}\dv\ds
	= \int_0^T\int_{\R^3}\big(\vv^{K+\theta}f_i^\eps\big)^{K/(K+\theta)}
	\big(\vv^K(f_i^\eps)^2\big)^{\theta/(K+\theta)}\dv\ds \\
	&\le C\int_0^T\|\vv^{K+\theta}f_i^\eps\|_{L^1(\R^3)}\ds
	+ C\int_0^T\|\vv^K(f_i^\eps)^2\|_{L^1(\R^3)}\ds \le C.
\end{align*}
Together with the uniform bounds for $f_i^\eps$ from Lemma \ref{lem.unifeps},
this yields a uniform bound for $\pa_t f_i^\eps$ in
$L^{r}(0,T;W^{-1,p}(\R^3))$ for some $r>1$, finishing the proof.
\end{proof}

The bounds of Lemmas \ref{lem.unifeps} and \ref{lem.unifeps2} and the compact
embedding $W^{1,p}(\R^3)\cap L^2(\R^3;\vv^K\dv)$ $\hookrightarrow L^2(\R^3)$
(see Lemma \ref{lem.comp} in Appendix \ref{sec.comp}) 
allow us to apply the Aubin--Lions lemma
to conclude the existence of a subsequence (not relabeled) such that, as $\eps\to 0$,
$$
  f_i^\eps\to f_i\quad\mbox{strongly in }L^2(0,T;L^2(\R^3)).
$$
Furthermore, we obtain weak convergences for $\na f_i^\eps$ and $\pa_t f_i^\eps$
in suitable spaces. At this point, it is straightforward to pass to the limit $\eps\to 0$
in \eqref{3.approx} to infer that $f_i^\delta:=f_i$ is a weak solution to
\begin{align}\label{4.approx}
  \pa_t f_i^\delta &+ \delta\bigg(\vv^K f_i^\delta - g(v)\int_{\R^3}\vv^K f_i^\delta \dv
	\bigg) - \delta\diver\big(|\na f_i^\delta|^{p-2}\na f_i^\delta\big) \\
	&= \sum_{j=1}^s c_{ji}[f^\delta]\diver\bigg(\na f_i^\delta 
	+ \frac{m_if_i^\delta}{T_{ji}[f^\delta]}
	(v-u_{ji}[f^\delta])\bigg)\quad\mbox{in }\R^3,\ t>0. \nonumber
\end{align}
We observe that the collision operator on the right-hand side is identical to that
one in \eqref{1.eq} and in particular, it conserves mass, momentum, and energy;
see Lemma \ref{lem.cons}.


\subsection{Limit $\delta\to 0$}

Let $f_i^\delta$ be the solution to \eqref{1.ic} and \eqref{4.approx}, constructed
in the previous subsection.
To perform the limit $\delta\to 0$, we derive some estimates uniform in $\delta$.
First, we note that mass conservation still holds, i.e.
$\|f_i^\delta\|_{L^1(\R^3)}= n_i$ for $i=1,\ldots,s$.

\begin{lemma}\label{lem.delta}
There exists a constant $C>0$ independent of $\delta$ (but depending on the initial data)
such that
\begin{align*}
  \sup_{0<t<T}\sum_{i=1}^s\int_{\R^3}\big(f_i^\delta(t)\log f_i^\delta(t) 
	+ f_i^\delta(t)|v|^2\big)\dv &\le C, \\
	\sum_{i,j=1}^s\int_0^T\int_{\R^3} c_{ji}[f^\delta]f_i^\delta
	\bigg|\na\log\frac{f_i^\delta}{M_{ij}[f^\delta]}\bigg|^2\dv\ds &\le C, \\
	\delta\sum_{i=1}^s\int_0^t\int_{\R^3}|\na(f_i^\delta)^{(p-1)/p}|^p\dv\ds 
	+ \delta\sum_{i=1}^s\int_0^t\int_{\R^3}\vv^{K+2}f_i^\delta\dv\ds &\le C.
\end{align*}
\end{lemma}

\begin{proof}
We split the proof in several parts.

{\em Step 1: Test function $\log f_i^\delta$.}
We use the test function $\log f_i^\delta$ in \eqref{4.approx}. Again, strictly speaking,
this test function cannot be used since we cannot exclude that $f_i^\delta=0$. 
We show in Appendix \ref{sec.app} how this argument
can be made rigorous. We obtain from formulation \eqref{1.gradflow} and
property \eqref{1.QlogM}
\begin{align}\label{4.I1011}
  &\sum_{i=1}^s\int_{\R^3} f_i^\delta(t)\log f_i^\delta(t)\dv
	- \sum_{i=1}^s\int_{\R^3} f_i^0\log f_i^0\dv
	+ \delta c_p\sum_{i=1}^s\int_0^t\int_{\R^3}
	|\na(f_i^\delta)^{(p-1)/p}|^p\dv\ds \\
	&\phantom{xxx}{}+ \sum_{i,j=1}^s\int_0^t\int_{\R^3} c_{ji}[f^\delta]f_i^\delta
	\bigg|\na\log\frac{f_i^\delta}{M_{ij}[f^\delta]}\bigg|^2\dv\ds \nonumber \\
	&\phantom{x}\le -\delta\sum_{i=1}^s\int_0^t\int_{\R^3}\vv^K f_i^\delta\log f_i^\delta\dv\ds
	\nonumber \\
	&\phantom{xxx}{}+ \delta\sum_{i=1}^s\int_0^t\bigg(\int_{\R^3}\log f_i^\delta g(v)\dv\bigg)
	\bigg(\int_{\R^3}\vv^K f_i^\delta\dv\bigg)\ds =: I_{10}+I_{11}. \nonumber
\end{align}
By mass conservation,
$$
  \int_{\R^3}\log f_i^\delta g(v)\dv\le \int_{\{f_i^\delta\ge 1\}}\log f_i^\delta g(v)\dv
	\le C \int_{\{f_i^\delta\ge 1\}}(1+f_i^\delta)g(v)\dv \le C,
$$
and consequently,
$$
  I_{11} \le C\delta\sum_{i=1}^s\int_0^t\int_{\R^3}\vv^K f_i^\delta\dv\ds
	\le C\delta + \frac{\delta}{32}\sum_{i=1}^s\int_0^t\int_{\R^3}\vv^{K+2} f_i^\delta\dv\ds.
$$
The term $I_{10}$ can be written as
$$
  I_{10} \le \delta\sum_{i=1}^s\int_0^t\int_{\R^3}\vv^K f_i^\delta
	\bigg(\log\frac{1}{f_i^\delta}\bigg)^+ \dv\ds,
$$
recalling that $z^+=\max\{0,z\}$. We choose $0<\alpha<1/(K+2)$ and use the inequality
$\log z\le z^\alpha/\alpha$ for $z=1/f_i^\delta>1$ as well as Young's inequality to estimate
\begin{align*}
  \vv^K f_i^\delta\bigg(\log\frac{1}{f_i^\delta}\bigg)^+
	&= \vv^K\mathrm{1}_{\{f_i^\delta<1\}}f_i^\delta\log\frac{1}{f_i^\delta}
	\le \frac{1}{\alpha}\vv^K (f_i^\delta)^{1-\alpha} \\
	&= \alpha^{-1}\vv^{-1}\big(\vv^{K+1}(f_i^\delta)^{1-\alpha}\big)
	\le \alpha^{-1/\alpha}\vv^{-1/\alpha} + \vv^{(K+1)/(1-\alpha)}f_i^\delta.
\end{align*}
It follows from $K>1$ that $-1/\alpha<-(K+2)<-3$ and hence, the integral over $\vv^{-1/\alpha}$
is finite. This yields, since $(K+1)/(1-\alpha)<K+2$,
$$
  I_{10} \le C\delta + \delta\sum_{i=1}^s\int_0^t\int_{\R^3}
	\vv^{(K+1)/(1-\alpha)}f_i^\delta\dv\ds
	\le C\delta + \frac{\delta}{32}\sum_{i=1}^s\int_0^t\int_{\R^3}\vv^{K+2}f_i^\delta\dv\ds.
$$
We insert the estimate for $I_{10}$ and $I_{11}$ into \eqref{4.I1011} to find that
\begin{align}\label{4.flogf}
  \sum_{i=1}^s&\int_{\R^3} f_i^\delta(t)\log f_i^\delta(t)\dv
	+ \delta c_p\sum_{i=1}^s\int_0^t\int_{\R^3}|\na(f_i^\delta)^{(p-1)/p}|^p\dv\ds \\
	&\phantom{xx}{}+ \sum_{i,j=1}^s\int_0^t\int_{\R^3} c_{ji}[f^\delta]f_i^\delta
	\bigg|\na\log\frac{f_i^\delta}{M_{ij}[f^\delta]}\bigg|^2\dv\ds \nonumber \\
	&\le C + \frac{\delta}{16}\sum_{i=1}^s\int_0^t\int_{\R^3}
	\vv^{K+2} f_i^\delta \dv\ds. \nonumber
\end{align}
We need to estimate the right-hand side.

{\em Step 2: Test function $|v|^2$.}
We use the test function $|v|^2$ (more precisely a suitable cutoff function, see
Appendix \ref{sec.app}) in \eqref{4.approx}. Since the collision operator conserves
the energy (see Lemma \ref{lem.cons}), the corresponding integral vanishes,
and we end up with
\begin{align*}
  \sum_{i=1}^s&\int_{\R^3}f_i^\delta(t)|v|^2\dv - \sum_{i=1}^s\int_{\R^3}f_i^0|v|^2\dv
	+ \delta\sum_{i=1}^s\int_0^t\int_{\R^3}\vv^K|v|^2 f_i^\delta\dv\ds \\
	&= \sum_{i=1}^s\int_0^t\bigg(\int_{\R^3}|v|^2g(v)\dv\bigg)
	\bigg(\int_{\R^3}\vv^K f_i^\delta\dv\bigg)\ds \\
	&\phantom{xx}{}
	- 2\delta\sum_{i=1}^s\int_0^t\int_{\R^3}|\na f_i^\delta|^{p-2}\na f_i^\delta\cdot v\dv\ds \\
	&\le C(\delta) + \frac{\delta}{8}\sum_{i=1}^s\int_0^t\int_{\R^3}\vv^{K+2} f_i^\delta\dv\ds
	+ 2\delta\sum_{i=1}^s\int_0^t\int_{\R^3}|\na f_i^\delta|^{p-1}|v|\dv\ds.
\end{align*}
Since $\vv^K|v|^2 = \vv^{K+2}-\vv^K\ge \frac12\vv^{K+2}-C$, 
the last term on the left-hand side is bounded from below by
\begin{align*}
  \delta\sum_{i=1}^s\int_0^t\int_{\R^3}\vv^K|v|^2 f_i^\delta\dv\ds
	&\ge \frac{\delta}{2}\sum_{i=1}^s\int_0^t\int_{\R^3}\vv^{K+2} f_i^\delta\dv\ds
	- C\delta \sum_{i=1}^s\int_0^t\int_{\R^3}f_i^\delta\dv\ds \\
	&\ge \frac{\delta}{2}\sum_{i=1}^s\int_0^t\int_{\R^3}\vv^{K+2} f_i^\delta\dv\ds - C\delta,
\end{align*}
where we used again mass conservation in the last step. Therefore,
\begin{equation}\label{4.aux3}
  \sum_{i=1}^s\int_{\R^3}f_i^\delta(t)|v|^2\dv
	+ \frac{3\delta}{8}\sum_{i=1}^s\int_0^t\int_{\R^3}\vv^{K+2} f_i^\delta\dv\ds 
	\le C + 2\delta\sum_{i=1}^s\int_0^t\int_{\R^3}
	|\na f_i^\delta|^{p-1}|v|\dv\ds.
\end{equation}

We estimate the term on the right-hand side of \eqref{4.aux3}. Let $q>1$.
We apply Young's inequality twice with exponents $(p,p/(p-1))$ and $(q,q/(q-1))$:
\begin{align}
  2\delta\int_{\R^3}&|\na f_i^\delta|^{p-1}|v|\dv
	\le C\delta\int_{\R^3}\big(|v||f_i^\delta|^{(p-1)/p}\big)
	|\na(f_i^\delta)^{(p-1)/p}|^{p-1}\ds \nonumber \\ 
	&\le C\delta\int_{\R^3}|v|^p|f_i^\delta|^{p-1}\dv 
	+ \frac{\delta}{4}c_p\int_{\R^3}|\na(f_i^\delta)^{(p-1)/p}|^p\dv \label{4.aux4} \\
	&\le \delta\int_{\R^3}\bigg(\frac{q-1}{q}\big(C|v|^{p}(f_i^\delta)^{1-1/q}\big)^{q/(q-1)}
	+ \frac{1}{q}(f_i^\delta)^{(p-2+1/q)q}\bigg)\dv \nonumber \\
	&\phantom{xx}{}+ \frac{\delta}{4}c_p\int_{\R^3}|\na(f_i^\delta)^{(p-1)/p}|^p\dv 
	\nonumber v\\
	&\le C\delta\int_{R^3}|v|^{pq/(q-1)}f_i^\delta\dv
	+ \frac{\delta}{q}\int_{\R^3}(f_i^\delta)^{1+q(p-2)}\dv 
	+ \frac{\delta}{4}c_p\int_{\R^3}|\na(f_i^\delta)^{(p-1)/p}|^p\dv, \nonumber 
\end{align}
where $c_p>0$ is as in \eqref{4.flogf}.
We deduce from the Gagliardo--Nirenberg inequality that
\begin{align*}
  & \|\psi\|_{L^r(\R^3)}\le C\|\na\psi\|_{L^p(\R^3)}^\theta
	\|\psi\|_{L^{p/(p-1)}(\R^3)}^{1-\theta}, \quad\mbox{where} \\
  & r = \frac{p}{p-1}(1+q(p-2)), \quad \theta = \frac{3q(p-1)(p-2)}{2(2p-3)(1+q(p-2))},
\end{align*}
applied to $\psi=(f_i^\delta)^{(p-1)/p}$, that
\begin{align*}
  \frac{\delta}{q}\int_{\R^3}(f_i^\delta)^{1+q(p-2)}\dv
	&= \frac{\delta}{q}\|(f_i^\delta)^{(p-1)/p}\|_{L^r(\R^3)}^r
	\le C\delta\|\na(f_i^\delta)^{(p-1)/p}\|_{L^p(\R^3)}^{r\theta}
	\|f_i^\delta\|_{L^1(\R^3)}^{(p-1)(1-\theta)/p} \\
	&\le C\delta\|\na(f_i^\delta)^{(p-1)/p}\|_{L^p(\R^3)}^{r\theta}
	\le \frac{\delta}{4}c_p\|\na(f_i^\delta)^{(p-1)/p}\|_{L^p(\R^3)}^p + C\delta,
\end{align*}
where we used mass conservation in the last but one step and the fact $r\theta<p$
as well as Young's inequality in the last step.
Choosing $q=4/3$, the first term on the 
right-hand side of \eqref{4.aux4} is estimated according to
\begin{align*}
  C\delta\int_{R^3}|v|^{pq/(q-1)}f_i^\delta\dv
	= C\delta\int_{R^3}|v|^{4p}f_i^\delta\dv
	\le \frac{\delta}{4}\int_{\R^3}\vv^{K+2}f_i^\delta\dv + C\delta, 
\end{align*}
if we choose $K>4p-2$ so that $4p<K+2$. We conclude from \eqref{4.aux4} that
\begin{align*}
  2\delta\int_{\R^3}&|\na f_i^\delta|^{p-1}|v|\dv
	\le C\delta + \frac{\delta}{4}\int_{\R^3}\vv^{K+2}f_i^\delta\dv
	+ \frac{\delta}{2}c_p\|\na(f_i^\delta)^{(p-1)/p}\|_{L^p(\R^3)}^p
\end{align*}
and then from \eqref{4.aux3} that
\begin{align*}
  \sum_{i=1}^s\int_{\R^3}f_i^\delta|v|^2\dv
	&+ \frac{\delta}{8}\sum_{i=1}^s\int_0^t\int_{\R^3}\vv^{K+2} f_i^\delta\dv\ds \\
	&\le C + \frac{\delta}{2}c_p\sum_{i=1}^s\int_0^t\int_{\R^3}
	|\na(f_i^\delta)^{(p-1)/p}|^p\dv\ds.
\end{align*}

{\em Step 3: End of the proof.}
We add the previous inequality to \eqref{4.flogf}:
\begin{align*}
  &\sum_{i=1}^s\int_{\R^3}\big(f_i^\delta(t)\log f_i^\delta(t) + f_i^\delta(t)|v|^2\big)\dv
	+ \frac{\delta}{2}c_p\sum_{i=1}^s\int_0^t\int_{\R^3}|\na(f_i^\delta)^{(p-1)/p}|^p\dv\ds \\
	&\phantom{x}{}+ \frac{\delta}{16}\sum_{i=1}^s\int_0^t\int_{\R^3}\vv^{K+2}f_i^\delta\dv\ds
	+ \sum_{i,j=1}^s\int_0^t\int_{\R^3} c_{ji}[f^\delta]f_i^\delta
	\bigg|\na\log\frac{f_i^\delta}{M_{ij}[f^\delta]}\bigg|^2\dv\ds
	\le C.
\end{align*}
This concludes the proof.
\end{proof} 

The energy bound in Lemma \ref{lem.delta} shows 
that the temperature $T_i[f^\delta]$, defined in \eqref{1.mom}, is
bounded from above uniformly in $\delta$ and $(0,T)$. 
This implies that $c_{ji}[f^\delta]$, defined in \eqref{1.cji}, 
is bounded from above uniformly in $\delta$ and $(0,T)$ when $\gamma\geq 0$.
We claim that the temperature $T_{ji}[f^\delta]$ is also uniformly bounded from below, 
which implies that $c_{ji}[f^\delta]$ is bounded from above uniformly in $\delta$ and $(0,T)$ also when $\gamma < 0$.

\begin{lemma}\label{lem.lowerT}
There exists a constant $c>0$, only depending on the initial entropy 
(and in particular independent of $\delta$), such that 
$$
  \inf_{0<t<T}T_{ji}[f^\delta(t)] \ge c>0.
$$
\end{lemma}

\begin{proof}
Define $\Phi(x)=\mu(1+x)\log(1+x)-\mu x$ for $x\ge 0$, where $\mu>0$. Then
$\Phi^*(y)=\mu e^{y/\mu}-y-\mu$ for $y\ge 0$ is its convex conjugate, and the
Fenchel--Young inequality $xy\le \Phi(x)+\Phi^*(y)$ holds. We infer from the lower 
bound \eqref{4.Tieps} and the Fenchel--Young inequality with $x=f_i^\delta$ and $y=1$ that
\begin{align*}
  T_i[f^\delta] &\ge C\lambda^2\bigg(n_i - \int_{\{|v-u_i|\le\lambda\}}
	f_i^\delta\dv\bigg) \\
	&\ge C\lambda^2\bigg(n_i - \mu\int_{\R^3}(1+f_i^\delta)\log(1+f_i^\delta)\dv
	- \frac43\pi\mu e^{1/\mu}\lambda^3\bigg) \\
	&\ge C\lambda^2\bigg(n_i - \mu C_0 - \frac43\pi\mu e^{1/\mu}\lambda^3\bigg),
\end{align*}
since the volume of the ball in $\R^3$ with radius $\lambda$ equals $4\pi\lambda^3/3$,
and $C_0$ depends on the initial data via the first estimate in Lemma \ref{lem.delta}.
Then, choosing $\mu=1/\log(C_0\lambda^{-3})$, a computation reveals that
$$
  T_i[f^\delta] \ge C\lambda^2\bigg(n_i - \frac{C_1}{\log(C_0\lambda^{-3})}\bigg),
	\quad\mbox{where }C_1 = C_0\bigg(1+\frac43\pi\bigg).
$$
It follows from the choice $\lambda=[C_0\exp(-2C_1/n_i)]^{1/3}$ that
$T_i[f^\delta]\ge c>0$ for $c=C\lambda^2n_i/2$, and this inequality is uniform in $(0,T)$.
It can be seen from \eqref{4.Tieps} that $C$ is proportional to $1/n_i$ such that
the constant $c$ only depends on the initial entropy and energy via $C_0$. Consequently,
$T_{ji}[f^\delta]\ge\min\{T_i[f^\delta],T_j[f^\delta]\}\ge c>0$.
\end{proof}

\begin{remark}\rm 
Observe that the uniform positive bound on $T_{ji}[f^\delta]$ yields a uniform
bound for $c_{ji}[f^\delta]$ in $L^\infty(0,T)$ even in the case $\gamma<0$ so that
$c_{ji}[f^\delta]$ is uniformly bounded in $L^\infty(0,T)$ for any $\gamma\in\R$.
We can also conclude a uniform positive bound for $c_{ji}[f^\delta]$ 
for every $\gamma\in\R$.
\qed
\end{remark}

\begin{lemma}\label{lem.H1}
There exists a constant $C>0$ independent of $\delta$ such that
$$
  \inf_{[0,T]}c_{ji}[f^\delta]\geq C^{-1},\quad 
  \sup_{[0,T]}c_{ji}[f^\delta]\leq C,\quad 
  \|\na f_i^\delta\|_{L^2(0,T;L^1(\R^3))}\le C.
$$
\end{lemma}

\begin{proof}
The bounds for $c_{ji}[f^\delta]$ follow from definitions \eqref{1.mom} and \eqref{1.cji}
as well as Lemmas \ref{lem.delta} and \ref{lem.lowerT}.
By the second estimate in Lemma \ref{lem.delta} and the fact that 
$f_i^\delta|\na\log M_{ii}[f^\delta]|^2$ (which is bounded by the energy)
is uniformly bounded in $L^\infty(0,T;L^1(\R^3))$,
\begin{align*}
  \int_0^T&\int_{\R^3}c_{ji}[f^\delta]|\na(f_i^\delta)^{1/2}|^2\dv\ds
	= \frac14\int_0^T\int_{\R^3}c_{ji}[f^\delta]f_i^\delta|\na\log f_i^\delta|^2\dv\ds \\
	&\le \frac12\int_0^T\int_{\R^3}c_{ji}[f^\delta]\bigg(f_i^\delta
	\bigg|\na\log\frac{f_i^\delta}{M_{ij}}[f^\delta]\bigg|^2
  + f_i^\delta|\na\log M_{ij}[f^\delta]|^2\bigg)\dv\ds \le C.
\end{align*}
Consequently, by the Cauchy--Schwarz inequality,
\begin{align*}
  \int_0^T\bigg(\int_{\R^3} & c_{ji}[f^\delta]|\na f_i^\delta|\dv\bigg)^2\ds
	= 4\int_0^T c_{ji}[f^\delta]^2\bigg(\int_{\R^3}(f_i^\delta)^{1/2}
	|\na(f_i^\delta)^{1/2}|\dv\bigg)^2\ds \\
	&\le 4\int_0^Tc_{ji}[f^\delta]^2\|(f_i^\delta)^{1/2}\|_{L^2(\R^3)}^2
	\|\na(f_i^\delta)^{1/2}\|_{L^2(\R^3)}^2\ds \\
	&\le 4\sup_{0<t<T}\|f_i^\delta(t)\|_{L^1(\R^3)}\int_0^T\int_{\R^3}c_{ji}[f^\delta]^2
	|\na(f_i^\delta)^{1/2}|^2\dv\ds \le C.
\end{align*}
The lemma follows from the uniform lower bound for $c_{ji}[f^\delta]$.
\end{proof}

We claim that $\pa_t f_i^\delta$ is uniformly bounded in
$L^r(0,T;W^{-1,1}(\R^3))$ for some $r>1$. Indeed, by Lemma \ref{lem.delta} 
and Jensen's inequality \eqref{4.vv}, 
$\delta\vv^K f_i^\delta$ is uniformly bounded in $L^{(K+2)/K}(0,T;$ $L^{1}(\R^3))$
and $f_i^\delta(v-u_{ji}[f^\delta])$ is uniformly bounded in $L^\infty(0,T;L^1(\R^3))$.
Lemma \ref{lem.delta} also shows that $\delta|\na f_i^\delta|^{p-2}\na f_i^\delta$ is uniformly bounded in $L^{p/(p-1)}(0,T;L^{p/(p-1)}(\R^3))$ and by Lemma \ref{lem.H1}, 
$c_{ji}[f^\delta]\na f_i^\delta$ is uniformly bounded in $L^2(0,T;L^1(\R^3))$. 
This shows the claim with $r=\min\{(K+2)/K,p/(p-1),2\}$.

Since the embedding $W^{1,1}(\R^3)\cap L^1(\R^3;(1+|v|^2)\dv)
\hookrightarrow L^1(\R^3)$ is compact (the proof is similar to that one of Lemma
\ref{lem.comp}), we can apply the Aubin--Lions lemma to conclude
the existence of a subsequence (not relabeled) such that
$$
  f_i^\delta\to f_i\quad\mbox{strongly in }L^2(0,T;L^1(\R^3)).
$$
Furthermore, for a subsequence,
$$
  \pa_t f_i^\delta\rightharpoonup \pa_t f_i \quad\mbox{weakly in }L^r(0,T;W^{-1,1}(\R^3)),
$$
and $\delta\diver(|\na f_i^\delta|^{p-2}\na f_i^\delta)\to 0$ strongly in
$L^p(0,T;W^{-1,p}(\R^3))$. 

Next, we claim that 
$$
  \delta\vv^K f_i^\delta\to 0\quad\mbox{strongly in }L^1(0,T;L^1(\R^3)). 
$$
Indeed, the strong convergence of $f_i^\delta$ and the uniform bound for $\vv^{K+2}f_i^\delta$
show that, for any $R>0$,
\begin{align*}
  \limsup_{\delta\to 0}&\int_0^T\int_{\R^3}\delta\vv^K f_i^\delta\dv\ds
	= \limsup_{\delta\to 0}\int_0^T\bigg(\delta\int_{\{|v|\le R\}}\vv^K f_i^\delta\dv
	+ \delta\int_{\{|v|>R\}}\vv^K f_i^\delta\dv\bigg)\ds \\
	&= \limsup_{\delta\to 0}\int_0^T\int_{\{|v|>R\}}\vv^K f_i^\delta\dv\ds
	\le \frac{1}{R^2}\limsup_{\delta\to 0}\int_0^T\int_{\R^3}\vv^{K+2} f_i^\delta\dv\ds
	\le \frac{C}{R^2}.
\end{align*}
This yields $\limsup_{\delta\to 0}\int_0^T\int_{\R^3}\vv^K f_i^\delta\dv\ds=0$,
proving the claim.

The convergence $u_{ji}[f^\delta]\to u_{ji}[f]$ strongly in $L^q(0,T)$ for any $q<\infty$
follows from the uniform $L^\infty(0,T)$ bound of the energy 
$\sum_{i=1}^s\int_{\R^3}f_i^\delta|v|^2\dv$.
To show the convergence of the temperature $T_{ji}[f^\delta]$, we need a uniform bound
for a higher-order moment $\sum_{i=1}^s\int_{\R^3}f_i^\delta|v|^m\dv$ for some $m>2$.
This is done in a similar way as in Step 2 of Lemma \ref{lem.delta}, where we used the
test function $|v|^2$ in \eqref{4.approx}, but here we choose the test function $|v|^m$
with $m>2$. In this case, the collision operator gives a nonzero contribution, but
our previous estimates show that it is bounded, since $u_{ji}[f^\delta]$ is uniformly bounded
and $c_{ji}[f^\delta]$ and $T_{ji}[f^\delta]^{-1}$ are uniformly bounded from above.
This yields the existence of a constant $C>0$ such that
$$
  \sup_{0<t<T}\sum_{i=1}^s\int_{\R^3}\vv^m f_i^\delta(t)\dv \le C
	\quad\mbox{for some }m>2.
$$
It follows from this bound that $T_{ji}[f^\delta]\to T_{ji}[f]$ strongly in
$L^q(0,T)$ for every $q<\infty$. Now, we can pass to the limit $\delta\to 0$
in \eqref{4.approx}, showing that the limit function $f_i$ is a weak solution
to \eqref{1.eq}--\eqref{1.Tji}. 


\begin{appendix}

\section{A compactness result}\label{sec.comp}

\begin{lemma}\label{lem.comp}
The space $W^{1,p}(\R^3)\cap L^2(\R^3;(1+|v|^2)\dv)$ with $p>3$ is compactly embedded into
$L^2(\R^3)$ and in $L^\infty(\R^3)$.
\end{lemma}

\begin{proof}
The proof is inspired from \cite[Lemma 1]{CGZ20}. Let $(f_n)$ be bounded in 
$V:=W^{1,p}(\R^3)\cap L^2(\R^3;(1+|v|^2)\dv)$. It follows from the continuous embedding
$W^{1,p}(\R^3)\hookrightarrow L^\infty(\R^3)$ that there exists a subsequence, which
is not relabeled, such that $f_n\rightharpoonup f$ weakly in $L^\infty(\R^3)$ as $n\to\infty$.
Let $B_M\subset\R^3$ be the ball around the origin with radius $M>0$. Then, in view
of the compact embedding $W^{1,p}(B_M)\hookrightarrow L^\infty(B_M)$, up to a subsequence,
$f_n\to f$ strongly in $L^\infty(B_M)$. Thanks to a Cantor diagonal argument, the
subsequence $(f_n)$ can be chosen independent of $M$. By the uniform bound in $V$ and
Fatou's lemma, we have $f\in V$. Next, for sufficiently large $n\in\N$,
\begin{align*}
  \|f_n-f\|_{L^2(\R^3)} &= \int_{B_M}|f_n-f|^2\dv + \int_{\R^3\setminus B_M}|f_n-f|^2\dv \\
	&\le \frac{\eps}{2} + \frac{1}{M^2}\int_{\R^3}(1+|v|^2)|f_n-f|^2\dv \le \eps,
\end{align*}
if we choose also $M>0$ sufficiently large. 
Hence, $f_n\to f$ strongly in $L^2(\R^3)$. We use the Gagliardo--Nirenberg inequality
with $\beta=3p/(5p-6)\in(0,1)$:
$$
  \|f_n-f\|_{L^\infty(\R^3)} \le C\|\na(f_n-f)\|_{L^p(\R^3)}^\beta
	\|f_n-f\|_{L^2(\R^3)}^{1-\beta} \le C\|f_n-f\|_{L^2(\R^3)}^{1-\beta}\to 0
$$
as $n\to\infty$. This concludes the proof.
\end{proof}


\section{Rigorous test functions}\label{sec.app}

We have used $\vv^\theta$ for $\theta\ge 0$ and $\log f_i^\delta$ as test functions in
the corresponding weak formulations, which is not rigorous. To make the computations
rigorous, we need to approximate.
First, we introduce the cutoff functions
$$
  \psi_R(x) = \psi_1\bigg(\frac{x}{R}\bigg), \quad
	\psi_1(x) = \begin{cases}
	1 &\mbox{if }|x|<1, \\
	\frac12(1+\cos(\pi(|x|-1))) &\mbox{if }1\le |x|\le 2, \\
	0 &\mbox{if }|x|>2,
	\end{cases}
$$
and use $\vv^\theta\psi_R$ as a test function in \eqref{3.approx} (we take $\theta=0$
to verify the mass control). This leads to additional terms depending on $\psi_R$ and
$\na\psi_R$. We focus our attention to the most delicate one and use H\"older's
inequality with exponents $p/(p-1)$ and $p$ as well as 
$|\na\psi_R(v)|\le C/R$ in $\{R<|v|<2R\}$ and $|\na\psi_R|=0$ else:
\begin{align*}
  \int_{\R^3}&|\na f_i^\eps|^{p-1}|\na\psi_R|\vv^\theta\dv
	\le \frac{\delta}{4}\int_{\R^3}|\na f_i^\eps|^p\dv
	+ C(\delta)\int_{\R^3}|\na\psi_R|^p\vv^{p\theta}\dv \\
	&\le \frac{\delta}{4}\int_{\R^3}|\na f_i^\eps|^p\dv
	+ \frac{C(\delta)}{R^p}\int_{\{|v|<2R\}}\vv^{p\theta}\dv
	\le \frac{\delta}{4}\int_{\R^3}|\na f_i^\eps|^p\dv + C(\delta)R^{-p+p\theta+3},
\end{align*}
and the last term vanishes as $R\to\infty$ since we have chosen $0<\theta<1-3/p$.

Second, we use the test function $\log(f_i^\delta+\eta)-\log\eta$ for $0<\eta<1$ in
\eqref{4.approx}. For this, we observe that, by \eqref{1.QlogM},
\begin{align*}
  \sum_{i,j=1}^s&\int_{\R^3}c_{ij}[f^\delta]f_i^\delta\na\log\frac{f_i^\delta}{M_{ii}[f^\delta]}
	\cdot\na\log(f_i^\delta+\eta)\dv \\
	&= \sum_{i,j=1}^s\int_{\R^3}c_{ij}[f^\delta]f_i^\delta
	\bigg(1-\frac{\eta}{f_i^\delta+\eta}\bigg)\na\log\frac{f_i^\delta}{M_{ij}[f^\delta]}
	\cdot\na\log f_i^\delta\dv \\
	&= \sum_{i,j=1}^s\int_{\R^3}c_{ij}[f^\delta]f_i^\delta
	\bigg|\na\log\frac{f_i^\delta}{M_{ij}[f^\delta]}\bigg|^2\dv \\
	&\phantom{xx}{}- \sum_{i,j=1}^s\int_{\R^3}c_{ij}[f^\delta]\frac{\eta}{f_i^\delta+\eta}
	\na\log\frac{f_i^\delta}{M_{ij}[f^\delta]}\cdot\na f_i^\delta\dv \\
	&= \sum_{i,j=1}^s\int_{\R^3}c_{ij}[f^\delta]f_i^\delta\bigg(1 - \frac{\eta}{f_i^\delta+\eta}
	\bigg)\bigg|\na\log\frac{f_i^\delta}{M_{ij}[f^\delta]}\bigg|^2\dv \\
	&\phantom{xx}{}-\eta\sum_{i,j=1}^s\int_{\R^3}c_{ij}[f^\delta]
	\frac{f_i^\delta}{f_i^\delta+\eta}\na\log\frac{f_i^\delta}{M_{ij}[f^\delta]}
	\cdot\na\log M_{ij}[f_i^\delta]\dv.
\end{align*}
Then we obtain from \eqref{4.approx}, putting all terms of order $\eta$ to the
right-hand side,
\begin{align}
  \sum_{i=1}^s&\int_{\R^3}
  \big((f_i^\delta(t)+\eta)\log(f_i^\delta(t)+\eta) - \eta\log\eta\big)
	\dv\nonumber\\
	&\phantom{xx}{} + \delta\sum_{i=1}^s\int_0^t\int_{\R^3}\vv^K 
	f_i^\delta\log(f_i^\delta(t)+\eta)\dv\ds \nonumber \\
	&\phantom{xx}{}
	+ \delta c_p\sum_{i=1}^s\int_0^t\int_{\R^3}|\na(f_i^\delta+\eta)^{(p-1)/p}|^p\dv\ds 
	\label{app.aux} \\
	&\phantom{xx}{}+ \sum_{i,j=1}^s\int_0^t\int_{\R^3}c_{ij}[f^\delta]f_i^\delta
	\bigg(1 - \frac{\eta}{f_i^\delta+\eta}\bigg)
	\bigg|\na\log\frac{f_i^\delta}{M_{ij}[f^\delta]}\bigg|^2\dv\ds \nonumber \\
	&= \sum_{i=1}^s\int_{\R^3}
	\big((f_i^0 + \eta)\log(f_i^0+\eta) - \eta\log\eta\big)\dv \nonumber \\
	&\phantom{xx}{}+ \delta\sum_{i=1}^s\int_0^t\bigg(\int_{\R^3}\vv^K f_i^\delta\dv\bigg)
	\bigg(\int_{\R^3} g(v)\log(f_i^\delta+\eta)\dv\bigg)\dv \nonumber \\
	&\phantom{xx}{}
	+ \eta\sum_{i,j=1}^s\int_0^t\int_{\R^3}c_{ij}[f^\delta]\frac{f_i^\delta}{f_i^\delta+\eta}
	\na\log\frac{f_i^\delta}{M_{ij}[f^\delta]}\na\log M_{ij}[f_i^\delta]\dv. \nonumber 
\end{align}
The second term on the right-hand side can be bounded because of mass 
conservation.
The last integral can be controlled by
\begin{align*}
  \eta\sum_{i,j=1}^s&\int_0^t\int_{\R^3}c_{ij}[f^\delta]\frac{f_i^\delta}{f_i^\delta+\eta}
	\na\log\frac{f_i^\delta}{M_{ij}[f^\delta]}\na\log M_{ij}[f_i^\delta]\dv \\
	&\le \frac{1}{2}\sum_{i,j=1}^s\int_0^t\int_{\R^3}c_{ij}[f^\delta]f_i^\delta
	\bigg|\na\log\frac{f_i^\delta}{M_{ij}[f^\delta]}\bigg|^2\dv \\
	&\phantom{xx}{}+ \frac{1}{2}\sum_{i,j=1}^s\int_0^t\int_{\R^3}c_{ij}[f^\delta]f_i^\delta
	\bigg|\frac{\eta}{f_i^\delta+\eta}\bigg|^2|\na\log M_{ij}[f^\delta]|^2\dv.
\end{align*}
The first term on the right-hand side is absorbed by the left-hand side of \eqref{app.aux}.
The function
$$
  G_\eta(v) = f_i^\delta(v)\bigg|\frac{\eta}{f_i^\delta(v)+\eta}\bigg|^2|
	\na\log M_{ij}[f^\delta](v)|^2
$$
is uniformly bounded by 
$0\le G_\eta\le f_i^\delta|\na\log M_{ij}[f^\delta]|\in L^1(0,T;L^1(\R^3))$,
and converges to zero a.e.\ in $\R^3\times(0,T)$. Therefore, by dominated convergence,
$G_\eta\to 0$ strongly in $L^1(0,T;L^1(\R^3))$. Fatou's lemma allows us to perform
the limit $\eta\to 0$ in \eqref{app.aux}. Then, proceeding as in Step 1 of the
proof of Lemma \ref{lem.delta}, we derive the entropy inequality.
\end{appendix}


\end{document}